\documentclass[11pt,a4paper]{article}

\usepackage{graphicx}
\usepackage{amsmath}
\usepackage{amsfonts}
\usepackage{amssymb}
\usepackage{enumerate}
\usepackage{lscape}
\usepackage{longtable}
\usepackage{color}
\usepackage{url}

\newtheorem{theorem}{Theorem}[section]

\newtheorem{algorithm}{Algorithm}[section]

\newtheorem{corollary}{Corollary}[section]

\newtheorem{definition}{Definition}[section]

\newtheorem{lemma}{Lemma}[section]

\newtheorem{remark}{Remark}[section]

\newtheorem{assumption}{Assumption}[section]

\newenvironment{proof}[1][Proof]{\textbf{#1.} }{\ \rule{0.5em}{0.5em} \vspace{1ex}}

\date{May 4, 2012}

\DeclareMathOperator{\argmax}{argmax}

\newcommand{\OOO}{\mathcal{O}}
\newcommand{\RRR}{\mathcal{R}}
\newcommand{\PPP}{\mathcal{P}}

\newcommand{\calM}{\mathcal{M}}

\renewcommand{\iota}{{l}}
\newcommand{\card}{{\rm card}}

\newcommand{\qua}{m}%a funï¿œï¿œï¿œï¿œo quadratica
\newcommand{\quaa}{m}%a funï¿œï¿œï¿œï¿œo quadratica no Cap2
\newcommand{\fo}{f}%a funï¿œï¿œï¿œï¿œo objectivo
\newcommand{\frau}{g}%a funï¿œï¿œï¿œï¿œo que se usa no 4.1.
\newcommand{\drau}{\mathcal{D}}%o dominio das funï¿œï¿œï¿œï¿œes que se usam no 4.1. e por aï¿œï¿œ a diante
\newcommand{\vad}{x}%as variaveis em R^n
\newcommand{\pon}{y}%os pontos em R^n
\newcommand{\Pon}{Y}%o conjunto dos pontos em R^n
\newcommand{\coe}{\alpha}%os coeficientes no Cap4
\newcommand{\tay}{T}%a aproximaï¿œï¿œï¿œï¿œo de taylor
\newcommand{\dimquah}{(n+1)(n+2)/2}%dimensï¿œï¿œo do espaï¿œï¿œo das quadraticas em R^n horizontal
\newcommand{\nstep}{k}%a aproximaï¿œï¿œï¿œï¿œo de taylor

\newcommand{\RR}{\mathbb{R}}

\begin{document}

\title{Computation of Sparse Low Degree Interpolating Polynomials and their Application to Derivative-Free Optimization}
\author{
A. S. Bandeira\thanks{Program in Applied and Computational Mathematics,
Princeton University, Princeton, NJ 08544, USA ({\tt ajsb@math.princeton.edu}).}
\and
K. Scheinberg\thanks{
Department of Industrial and Systems Engineering, Lehigh University,
Harold S. Mohler Laboratory, 200 West Packer Avenue, Bethlehem, PA 18015-1582, USA
({\tt katyas@lehigh.edu}). The work of this author is partially supported by  AFOSR  under grant  FA9550-11-1-0239.
}
\and
L. N. Vicente%
\thanks{CMUC, Department of Mathematics, University of Coimbra,
3001-454 Coimbra, Portugal ({\tt lnv@mat.uc.pt}). Support for
this author was provided by FCT under grant PTDC/MAT/098214/2008.}
}
%\date{}
\maketitle
\footnotesep=0.4cm
{\small
\begin{abstract}
Interpolation-based trust-region methods are an important class
of algorithms for
Deriva- tive-Free Optimization which rely on locally approximating an
objective function by quadratic polynomial interpolation models, frequently
built from less points than there are basis components.

Often, in practical applications, the contribution of the problem variables
to the objective function is such that many pairwise correlations between
variables are negligible, implying, in the smooth case,
a sparse structure in the Hessian matrix.
To be able to exploit  Hessian sparsity, existing optimization approaches require the knowledge of the sparsity structure. The goal of this paper is to develop and analyze a method where the sparse models are constructed automatically.

The sparse recovery theory developed recently in the
field of compressed sensing characterizes conditions under which a sparse
vector can be accurately recovered from few random measurements.
Such a recovery is achieved by minimizing the $\ell_1$-norm of a vector subject to the measurements constraints.
We suggest an approach for building sparse quadratic
polynomial interpolation models
 by minimizing the $\ell_1$-norm of the entries
of the model Hessian subject to the interpolation conditions.
We show that this procedure recovers
accurate models when the function Hessian is sparse, using
relatively few randomly selected sample points.

Motivated by this result, we developed a practical interpolation-based
trust-region method using deterministic sample sets and minimum $\ell_1$-norm
quadratic models. Our computational results show
 that the new approach exhibits a promising numerical
performance both in the general case and in the sparse one.
\end{abstract}

\bigskip

\begin{center}
\end{center}
}

%%%%%%%%%%%%%%%%%%%%%%%%%%%%%%%%%%%%%%%%%%%%%%%%%%%%%%%%%%%%%%%%%%%%%%%

\section{Introduction}\label{chapter1}

\label{sec:cap1}

The wide range of applications of mathematical optimization
have been recently enriched by the developments in emerging
areas such as Machine Learning and Compressed Sensing, where a structure of a model needs to be recovered from some observations. Specially designed optimization methods have been developed to handle the new applications that  often give rise
to large scale, but convex and well structured problems.
However, in many real-world applications, the objective function is
calculated by some costly black-box simulation which does not provide information about
its derivatives. Although one could estimate the derivatives, e.g., by finite differences, such a process is often too expensive and can produce misleading results in the presence of noise. An alternative is to consider methods that do
not require derivative information, and such methods are the subject of
study in Derivative-Free Optimization (DFO).  In this paper we propose a reverse relationship between optimization and compressed sensing --- instead of using optimization methods
to solve compressed sensing problems
we use the results of compressed sensing to improve optimization methods by recovering and exploiting possible structures of the black-box objective functions.

An important class of methods in DFO are interpolation-based trust-region methods. At each iteration, these methods build a model of the objective function
that locally approximates it in some {\it trust region} centered at the current
 iterate. The model is then minimized in the trust region, and the
corresponding minimizer is, hopefully, a better candidate for being a minimizer of the objective function in the trust region, and thus, possibly, is taken as the next iterate. It is usually preferable that minimization of the model in the trust region is an easy task, hence the models should be simple.
The simplest yet meaningful class of models  is the class
of  linear functions. Their drawback is that they do not capture the curvature of the objective function and thus slow down the convergence of the methods.
A natural and convenient non-linear class of models, which is often efficiently used, is the quadratic class.
 Determined quadratic interpolation requires
sample sets whose  cardinality is approximately
equal to the square of the dimension,
which may turn out to be too costly if the objective function is expensive to
evaluate. An alternative is to consider underdetermined quadratic models, using
sample sets of  smaller size than the ones needed for determined
interpolation. However, in this case, the quality of the model may deteriorate.

In many applications, the objective function has structure,
such as sparsity of the Hessian, which one may
exploit to improve the efficiency of an optimization method.
In DFO, since derivatives are not known, typically neither is their sparsity structure. If the structure is known in advance,
such as in group partial separability, it can be exploited as it is proposed in~\cite{BColson_PhLToint_2005}.
The main idea of our work is to implicitly and automatically take advantage of the  sparsity of the Hessian
in the cases when the sparsity structure is not known in advance,  to build accurate
 models from relatively small sample sets. This goal is achieved by minimizing the $\ell_1$-norm of the Hessian model coefficients.

Our work relies on the sparse solution recovery theory developed recently in the field of compressed sensing, where one characterizes conditions under which a sparse signal can be accurately recovered from few random measurements. Such type of recovery is achieved by minimizing the
$\ell_1$-norm of the unknown signal subject to measurement constraints
and can be accomplished in polynomial time.

The contribution of this paper is twofold. First, we show that it is possible to
compute fully quadratic models (i.e., models with the same accuracy as
 second order Taylor models) for functions defined on $\RR^n$ with sparse Hessians
 based on randomly selected sample sets  with only $\OOO(n (\log n)^4)$
 sample points (instead of the $\OOO(n^2)$ required for the determined quadratic case)
when the number of non-zero elements in the Hessian of the
function is $\OOO(n)$. Second, we introduce a practical interpolation-based trust-region DFO
algorithm exhibiting  competitive numerical performance.

The
 state-of-the-art approach is to build quadratic interpolation models, based
on sample sets of any size
between $n+1$ and $(n+1)(n+2)/2$, taking up the available degrees of freedom
by choosing the models with the smallest Frobenius norm
of the Hessian~\cite{ARConn_KScheinberg_LNVicente_2009} or Hessian change~\cite{MJDPowell_2004}, and this
approach has been shown to be robust and efficient in practice
(see also the recent paper~\cite{SGratton_PLToint_ATro_2010} where the models are always determined,
varying thus the number of basis components). In the approach
proposed in this paper, the degrees of freedom are taken up by minimizing the
$\ell_1$-norm of the Hessian of the model. We have tested the practical DFO algorithm
using both minimum Frobenius and minimum $\ell_1$-norm models.
Our results demonstrate the ability of the $\ell_1$-approach to improve the results of the
Frobenius one in the presence of some form of sparsity in the
Hessian of the objective function.

This paper is organized as follows. In Section~\ref{chapter2}, we introduce
background material on interpolation models. We give a brief introduction
to compressed sensing in Section~\ref{chapter3}, introducing also
concepts related to {\it partially} sparse recovery (the details are left to a separate
paper~\cite{ABandeira_KScheinberg_LNVicente_2011-b}).
In Section~\ref{chapter4}, we
obtain the main result mentioned above for sparse recovery of models for
functions with sparse Hessians, using an orthogonal basis for the space
of polynomials of degree $\leq 2$.
The proof of this result is based on
sparse bounded orthogonal expansions which are briefly described in the
beginning of Section~\ref{chapter4}. In Section~\ref{chapter5}, we introduce our
practical interpolation-based trust-region method and present numerical results
for the two underdetermined quadratic model variants, defined by
minimum Frobenius and $\ell_1$-norm minimization.
Finally, in Section~\ref{chapter6} we draw some conclusions and discuss
possible avenues for future research.

The paper makes extensive use of vector, matrix, and functional norms.
We will use $\ell_p$ or $\|\cdot\|_p$ for vector and matrix norms, without ambiguity.
The notation~$B_p(x;\Delta)$ will represent a closed ball
in~$\RR^n$, centered at~$x$ and of radius~$\Delta$, in the $\ell_p$-norm,
i.e., $B_p(x;\Delta) = \{ y \in \RR^n: \|y-x\|_p \leq \Delta \}$.
For norms of functions on normed spaces $L$, we will use~$\|\cdot\|_L$.

%%%%%%%%%%%%%%%%%%%%%%%%%%%%%%%%%%%%%%%%%%%%%%%%%%%%%%%%%%%%%%%%%%%%%%%%%%%%%%%%%%%%%%%%%%%%%%%%%%%%%%%%%%%%%%%%%%%%%%%%%%%%%%%%%%%%%%%%%%%%%%%%%%%%%%%%%%%%%%%%
\section{Use of models in DFO trust-region methods}\label{chapter2}

\subsection{Fully linear and fully quadratic models}\label{sec:models}

One of the main techniques used in DFO consists
 of locally modeling the objective function $\fo:D\subset\RR^n\to\RR$
 by models that are ``simple'' enough to be optimized easily and
sufficiently ``complex'' to approximate $\fo$ well.
 If a reliable estimate of the derivatives  of the function is available, then
 one typically uses Taylor approximations of first and second order
as polynomial models of $\fo(x)$. In DFO one has no access to
derivatives or their accurate estimates, and hence other model classes
are considered. However, the essential approximation quality of the
Taylor models is required to be sustained when necessary by the models used in
convergent  DFO frameworks.
For instance, the simplest first order approximation is
provided by, the so-called, fully linear models, whose definition requires
$\fo$ to be smooth up to the first order.

\begin{assumption}\label{assumptionlinear}
Assume that $\fo$ is continuously differentiable with Lipschitz
continuous gradient (on an open set containing $D$).
(For simplicity we will assume that all the balls and neighborhoods considered
in this paper are contained in~$D$.)
\end{assumption}

The following definition is essentially the same as
given in~\cite[Definition 6.1]{ARConn_KScheinberg_LNVicente_2009}  stated
using balls in an arbitrary $\ell_p$-norm, with $p\in(0,+\infty]$.

\begin{definition} \label{deffullylinearmodel}
Let a function $f:D\to\RR$ satisfying Assumption~\ref{assumptionlinear} be
given. A set of model functions $ {\calM}=\{m: \mathbb{R}^n\to
\mathbb{R}, \, m \in C^1\}$  is called a {\it fully linear} class of
models if the following
hold:
\begin{enumerate}
\item
There exist positive constants  $\kappa_{ef}$, $\kappa_{eg}$,
and $\nu_1^m$, such that for any $x_0\in D$ and $ \Delta \in
(0,\Delta_{max}]$ there exists a model function $m$ in ${\cal
M}$, with Lipschitz continuous gradient
and corresponding
Lipschitz constant bounded by  $\nu_1^m$, and such that
\begin{itemize}

\item the error between the gradient of the model and the gradient
of the function satisfies $$\|
\nabla f(\vad) - \nabla m(\vad) \|_2 \; \leq \; \kappa_{eg} \, \Delta,
\quad \forall \vad \in B_p(x_0;\Delta),$$

\item and the error between the  model and the function satisfies $$|
f(\vad) - m(\vad) | \; \leq \;
\kappa_{ef} \, \Delta^2, \quad  \forall \vad \in B_p(x_0;\Delta).$$
\end{itemize}
Such a  model $m$  is called fully linear on $B_p(x_0; \Delta)$.
\item
For this class $\cal M$ there exists an algorithm, which we will
call a  `model-improvement' algorithm, \index{model!model-improvement algorithm}
that in a finite, uniformly
bounded \index{uniformly bounded!number of improvement steps}
(with respect to $x_0$ and $ \Delta$) number of steps can
\begin{itemize}
 \item either provide a certificate  \index{fully linear!certificate}
for a given  model $m\in \cal M$ that
 it is  fully linear on $B_p(x_0;\Delta)$,
\item or fail to provide such a certificate \index{fully linear!certificate}
and find a model $\tilde{m} \in \cal M$ fully linear on $B_p(x_0;\Delta)$.
\end{itemize}
\end{enumerate}

\end{definition}

It is important to note that this definition does not restrict fully
linear models to linear functions, but instead considers models that
approximate $\fo$ \emph{as well} as the linear Taylor approximations.
Linear models such as linear interpolation (and first order Taylor
approximation) do not capture the curvature information of the
function that they are approximating. To achieve better practical local
convergence rates in general it is essential to consider nonlinear
models. In this paper we focus on quadratic interpolation models,
which ultimately aim at a higher degree of approximation accuracy.
We call such approximation models fully quadratic,
following~\cite{ARConn_KScheinberg_LNVicente_2009}, and note that, as
in  the linear case, one can consider a wider class of models not
necessarily quadratic.
We now require the function $\fo$ to exhibit smoothness up to the second order.

\begin{assumption}\label{assumptionquadratica}
Assume that $\fo$ is twice differentiable with Lipschitz continuous
Hessian (on an open set containing $D$).
\end{assumption}

Below we state  the definition of fully quadratic models given
in~\cite[Definition 6.2]{ARConn_KScheinberg_LNVicente_2009}, again
using balls in an $\ell_p$-norm, with arbitrary $p\in(0,+\infty]$.
\begin{definition} \label{deffullyquadmodel}
Let a function $f:D\to\RR$ satisfying
Assumption~\ref{assumptionquadratica} be given. A set of model
functions $ {\cal M}=\{m: \mathbb{R}^n\to \mathbb{R}, \, m \in C^2
\}$  is called a {\it fully quadratic} class of models if the following hold:

\begin{enumerate}
\item
There exist positive constants  $\kappa_{ef}$, $\kappa_{eg}$, $\kappa_{eh}$,
and $\nu_2^m$, such that for any $x_0\in D$ and $ \Delta \in
(0,\Delta_{max}]$ there exists a model function $m$ in ${\calM}$,
with Lipschitz continuous Hessian
and corresponding
Lipschitz constant bounded by  $\nu_2^m$, and such that
\begin{itemize}
\item the error between the Hessian of the  model and the Hessian of the
function satisfies  $$\| \nabla^2
f(\vad) - \nabla^2 m(\vad) \|_2 \; \leq \; \kappa_{eh} \, \Delta, \quad
\forall \vad \in B_p(x_0;\Delta),$$

\item the error between the gradient of the model and the gradient
of the function satisfies $$\|
\nabla f(\vad) - \nabla m(\vad) \|_2 \; \leq \; \kappa_{eg} \, \Delta^2,
\quad \forall \vad \in B_p(x_0;\Delta),$$

\item and the error between the  model and the function satisfies $$|
f(\vad) - m(\vad) | \; \leq \;
\kappa_{ef} \, \Delta^3, \quad  \forall \vad \in B_p(x_0;\Delta).$$
\end{itemize}
Such a  model $m$  is called fully quadratic on $B_p(x_0; \Delta)$.
\item
For this class $\cal M$ there exists an algorithm, which we will
call a  `model-improvement' algorithm, \index{model!model-improvement algorithm}
that in a finite, uniformly
bounded \index{uniformly bounded!number of improvement steps}
(with respect to $x_0$ and $ \Delta$) number of steps can
\begin{itemize}
 \item either provide a certificate  \index{fully linear!certificate}
for a given  model $m\in \cal M$ that
 it is  fully quadratic on $B_p(x_0;\Delta)$,
\item or fail to provide such a certificate \index{fully linear!certificate}
and find a model $\tilde{m} \in \cal M$ fully quadratic on $B_p(x_0;\Delta)$.
\end{itemize}
\end{enumerate}

\end{definition}

This  definition of a fully quadratic class requires
 that given a model from the class one can either prove that it is a fully
 quadratic model of $f$ on a given $B_p(x_0;\Delta)$, and for given
$\kappa_{ef}$, $\kappa_{eg}$, $\kappa_{eh}$, and $\nu_2^m$,
independent of $x_0$ and $\Delta$,
or provide such a model.
  It is shown in~\cite[Chapter~6]{ARConn_KScheinberg_LNVicente_2009} that
model-improvement algorithms exist for quadratic interpolation and
regression models. Hence quadratic interpolation models form a fully
quadratic class of models. They also exist for a fully linear class of
models, where
a model-improvement algorithm in~\cite{ARConn_KScheinberg_LNVicente_2009}
checks if a quadratic interpolation model is built  using a
well-poised set of at least $n+1$ interpolation points. To certify
that a model is fully quadratic, a model-improvement algorithm in~\cite{ARConn_KScheinberg_LNVicente_2009} requires that the set of the
interpolation points is
well poised and contains $(n+1)(n+2)/2$ points in the proximity of~$x_0$.
Thus, using a model-improvement algorithm  often  implies considerable
computational cost: it may be prohibitive to maintain sets of
$(n+1)(n+2)/2$ sample points near the current iterate due to the cost of obtaining the
function values and the dimension of the problem. Moreover,
verifying that the sample set is well poised may require a
factorization of a matrix with $(n+1)(n+2)/2$ rows and columns
resulting in $\mathcal{O}(n^6)$ complexity. For small~$n$, this additional cost
may be negligible, but it becomes substantial as $n$ grows beyond a few dozen.

In this paper we show that for any given function $f$, there exist
constants  $\kappa_{ef}$, $\kappa_{eg}$, $\kappa_{eh}$, and $\nu_2^m$
and the corresponding fully quadratic class of quadratic models ${\cal
M}$ for which, given~$x_0$ and~$\Delta$, we can construct a fully
quadratic model of $f$ on $B_p(x_0;\Delta)$ from ${\cal M}$, with high
probability, using, possibly, less than $(n+1)(n+2)/2$ sample points.

Note that Definition~\ref{deffullyquadmodel} requires the existence of an
algorithm which can {\it deterministically} certify that a given model
is fully quadratic. This requirement is imposed  because it
enables the deterministic convergence analysis of an algorithmic framework,
 provided in~\cite{ARConn_KScheinberg_LNVicente_2009-paper}
(see also~\cite[Chapter~10]{ARConn_KScheinberg_LNVicente_2009}), based on fully quadratic
(or fully linear) models. In contrast, in this
work, we consider an algorithm which cannot certify that a given model
is fully quadratic, but can construct such models with high
probability, hopefully, at a
considerable computational saving.
To adapt this approach in a convergent algorithmic framework, a
stochastic version of such a framework has to be designed and
analyzed.
 This work is a subject of future research (see~\cite{MCFerris_GDeng_2008} for some
 relevant theoretical results).

Finally, we want to point out that while the full convergence theory for
 the new model-based algorithmic framework is under development, the practical
implementation reported in this paper shows that in a simple trust-region framework the new
method works as well as or better than other methods discussed in~\cite{ARConn_KScheinberg_LNVicente_2009}.

\subsection{Quadratic polynomial interpolation models}\label{subsec:quad-models}

In model based DFO fully quadratic models of~$\fo$  are often obtained from
the class of  quadratic polynomials  by  interpolating~$\fo$ on some sample set of points $Y$. A detailed description of this process and related theory is given in \cite{ARConn_KScheinberg_LNVicente_2009}. Here we present
 briefly the basic ideas and necessary notation.

Let $\PPP_n^2$ be the space of polynomials of degree less than or equal to $2$ in $\RR^n$. The dimension of this space is $q=\dimquah$. A basis $\phi$ for
$\PPP_n^2$ will be denoted by $\phi=\{\phi_\iota(x)\}$ with $\iota=1, \ldots, q $. The most natural basis for polynomial spaces is the one consisting of the monomials, or the
\emph{canonical basis}. This basis appears naturally in Taylor models and is given for $\PPP_n^2$ by
\begin{equation}\label{basecanonicaigualdade}
\bar{\phi} \; = \; \left\{\frac12\vad_1^2,...,\frac12\vad_n^2,\vad_1\vad_2,...,\vad_{n-1}\vad_n,\vad_1,...,\vad_n,1 \right\}.
\end{equation}

We say that the quadratic function $\quaa$ interpolates $\fo$ at a given point $\pon$ if $\quaa(\pon)=\fo(\pon)$.
Assume that we are given a set $\Pon=\{\pon^1,...,\pon^p\}\subset\RR^n$ of interpolation points. A quadratic function $\quaa$ that interpolates $\fo$ at the points in $\Pon$, written as
$$\quaa(\vad) \; = \; \sum_{\iota=1}^q\alpha_\iota\phi_\iota(\vad),$$
must satisfy the following $p$ interpolation conditions $\sum_{\iota=1}^{q}\alpha_\iota\phi_\iota(\pon^i)=\fo(\pon^i),$ $i=1,...,p$.
These conditions form a linear system,
\begin{equation}\label{condinterpolacaomat}M(\phi,\Pon)\alpha \; = \; \fo(\Pon),\end{equation}
where $M(\phi,\Pon)$ is the interpolation matrix and $\fo(\Pon)_i=\fo(\pon^i),\ i=1,...,p$.

A sample set $\Pon$ is poised for (determined) quadratic interpolation if the corresponding interpolation matrix $M(\phi,\Pon)$ is square ($p$=$q$) and non-singular, guaranteeing that there exists a unique quadratic polynomial $\quaa$ such that $\quaa(\Pon)=\fo(\Pon)$. It is not hard to prove that this definition of poisedness and the uniqueness of the interpolant do not depend either on $\fo$ or on the basis $\phi$ (see~\cite[Chapter~3]{ARConn_KScheinberg_LNVicente_2009}).

In~\cite[Chapters~3~and~6]{ARConn_KScheinberg_LNVicente_2009} rigorous conditions on $\Pon$ are derived which ensure ``well poisedness'' for  quadratic interpolation. Under these conditions it is shown that
if $\Pon\subset B_2(x_0;\Delta)$ is a well-poised sample set for quadratic interpolation, then the quadratic function $\qua$ that interpolates $\fo$ on $\Pon$ is a fully quadratic model for $\fo$ on $B_2(x_0;\Delta)$
for some fixed positive constants $\kappa_{ef}$, $\kappa_{eg}$, $\kappa_{eh}$, and $\nu^m_2$.

One of the conditions imposed in~\cite{ARConn_KScheinberg_LNVicente_2009} on  a sample set $\Pon$ to guarantee fully quadratic interpolation model, is that
$\Pon$ has to contain $p=\dimquah$ points. However, building such a sample set costs $\dimquah$ evaluations of the function $\fo$
which  is too expensive for many applications. A typical efficient approach is to
consider smaller sample sets, which makes the linear system in~(\ref{condinterpolacaomat}) underdetermined.

\subsubsection{Underdetermined quadratic interpolation}\label{subsec-underdetermined}

We will now consider the case where the size of the sample set $\Pon$ satisfies $n+1<p<\dimquah$, in other words,
 when there are more points than is required for linear interpolation but fewer than is necessary for determined quadratic interpolation. If we consider the class of all quadratic functions that interpolate $\fo$ on
$\Pon$, then we can choose a model from this class that for one reason or other seems the most suitable. In particular approaches in~\cite{ARConn_KScheinberg_PhLToint_1998} and~\cite{SMWild_2008} select a quadratic model with the smallest possible Frobenius norm of the Hessian matrix, while in~\cite{MJDPowell_2004} a model is chosen to minimize the Frobenius norm of the {\em change} of the Hessian from one iteration to the next.
The former approach is studied in detail in~\cite[Chapter~5]{ARConn_KScheinberg_LNVicente_2009}. Let us introduce the basic ideas here.

To properly introduce the underdetermined models that we wish to consider
we split the basis $\bar{\phi}$ in~(\ref{basecanonicaigualdade}) into its linear and quadratic components: $\bar{\phi}_L=\{\vad_1,...,\vad_n,1\}$ and $\bar{\phi}_Q=\bar{\phi}\setminus \bar{\phi}_L$.

An essential property of  a sample set~$\Pon$ with $|\Pon|>n+1$ is that the matrix  $M(\bar{\phi}_L,\Pon)$ must have sufficiently linearly independent columns (in \cite[Section~4.4]{ARConn_KScheinberg_LNVicente_2009} it is said that~$\Pon$  is
well poised for linear regression). Roughly speaking, well poisedness  means that
 under a suitable scaling of $Y$,  $M(\bar{\phi}_L,\Pon)$ has a relatively small condition number, see~\cite[Section~4.4]{ARConn_KScheinberg_LNVicente_2009}. In that case for {\em any} quadratic  model which interpolates $\fo$ on $\Pon$ the following holds (see~\cite[Theorem~5.4]{ARConn_KScheinberg_LNVicente_2009} for a rigorous statement and proof).

\begin{theorem}\label{quadimplicalin}
For any $x_0\in D$ and $\Delta\in (0, \Delta_{max}]$, let $\qua$ be a quadratic function that interpolates $\fo$ in $\Pon$, where $\Pon\subset B_2(x_0;\Delta)$ is a sample set well poised for linear regression. Then, $\qua$ is fully linear (see Definition~\ref{deffullylinearmodel}) for $\fo$ on $B_2(x_0;\Delta)$ where the constants $\kappa_{ef}$ and $\kappa_{eg}$ are $\OOO(1+\|\nabla^2\qua\|_2)$ and depend on the condition number of
$M(\bar{\phi}_L,\Pon)$ (where $Y$ here is suitably scaled).
\end{theorem}

Theorem~\ref{quadimplicalin} suggests that one should build underdetermined quadratic models with ``small'' model Hessians, thus motivating minimizing
its Frobenius norm subject to~(\ref{condinterpolacaomat})
as in~\cite{ARConn_KScheinberg_PhLToint_1998} and~\cite{SMWild_2008}.
 Recalling the split of the basis $\bar{\phi}$ into the linear and the quadratic parts, one can write the interpolation model as
\[
\qua(\vad) \; = \; \coe_Q^T\bar{\phi}_Q(\vad) + \coe_L^T\bar{\phi}_L(\vad),
\]
where $\coe_Q$ and $\coe_L$ are the
corresponding  parts of the coefficient vector $\coe$. The minimum Frobenius norm solution~\cite[Section~5.3]{ARConn_KScheinberg_LNVicente_2009} can now be defined as the solution to the following optimization problem
\begin{equation}\label{minFrobenius}\begin{array}{cl}
\min & \frac12\|\coe_Q\|^2_2\\ [1ex]
\operatorname{s.t.} & M(\bar{\phi}_Q,\Pon)\coe_Q + M(\bar{\phi}_L,\Pon)\coe_L \; = \; \fo(\Pon).
\end{array}\end{equation}
(If $|\Pon|=\dimquah$ and $M(\bar{\phi}, \Pon)$ is nonsingular, this reduces to determined quadratic interpolation.)
Note that~(\ref{minFrobenius}) is a convex quadratic program with a closed form solution.

In~\cite[Section~5.3]{ARConn_KScheinberg_LNVicente_2009} it is shown
that under some additional conditions of well poisedness on~$\Pon$,
the minimum Frobenius norm (MFN) interpolating model can be fully linear with uniformly bounded error constants $\kappa_{ef}$ and $\kappa_{eg}$. Hence, the MFN
quadratic models provide at least as accurate interpolation as linear models.

On the other hand, it has not been shown so far that any class of underdetermined quadratic interpolation models provide provably better
  approximation of $\fo$ than fully linear models. The purpose of this paper is to show how to construct, with high probability,
underdetermined quadratic interpolation models that are fully quadratic.

\subsubsection{Sparse quadratic interpolation}\label{subsec-sparse}

It is clear that without any additional assumptions on $\fo$ we cannot guarantee a fully quadratic accuracy by an interpolation model based on less than $\dimquah$ points. We will thus consider the structure that is most commonly observed and exploited in large-scale derivative based optimization: the (approximate) sparsity of the Hessian of $\fo$.
Special structure, in particular group partial separability of $\fo$, has been exploited in DFO before, see~\cite{BColson_PhLToint_2005}. However, it was assumed that the specific structure is known in advance.
In the derivative-free setting, however, while the sparsity structure
of the Hessian may be known in some cases, it is often unavailable.
Moreover, we do not need to assume that there exists a fixed sparsity structure of the Hessian.

What we assume in this paper is that the Hessian of $\fo$ is ``approximately'' sparse in the domain where the model is built. In other words we assume the existence of a sparse fully quadratic model (a rigorous definition is provided in Section~\ref{sec:43}). In the case where $\nabla^2\fo$ is itself sparse, a Taylor expansion may serve as such a model.
 The main focus of our work is to recover sparse quadratic
models from the interpolation conditions.

Instead of solving (\ref{minFrobenius})  we construct quadratic models from the solution to the following optimization problem
\begin{equation}\label{minl1cap2}\begin{array}{cl}
\min & \|\coe_Q\|_1\\ [1ex]
\operatorname{s.t.} & M(\bar{\phi}_Q,\Pon)\coe_Q + M(\bar{\phi}_L,\Pon)\coe_L \; = \; \fo(\Pon),
\end{array}\end{equation}
where $\coe_Q$, $\coe_L$, $\bar{\phi}_Q$, and $\bar{\phi}_L$ are defined as in~(\ref{minFrobenius}).
 Solving~(\ref{minl1cap2})  is  tractable, since it is a linear program. Note that  minimizing
 the $\ell_1$-norm of the entries of the Hessian model indirectly controls  its $\ell_2$-norm
 and therefore  is an appealing approach from the perspective of
 Theorem~\ref{quadimplicalin}.  This makes the new approach a reasonable alternative to
building MFN models. As we will show in this paper, this approach is advantageous when the
 Hessian of $\fo$ has  zero entries (in other words, when there is no
 direct interaction between some of the variables of the objective function~$\fo$).
 In such cases, as we will show  in Section~\ref{chapter4}, we are able to recover,
 with high probability, fully quadratic models with much less than $\dimquah$ random points.
 This is the first result where a fully quadratic model is constructed from an underdetermined
 interpolation system. To prove this result we will rely on sparse vector recovery theory
 developed in the field of compressed sensing. In the next section we introduce the basic concepts and results that are involved.

\section{Compressed sensing}\label{chapter3}

Compressed sensing is a field concerned with the recovery of a sparse vector~$\bar{z}\in\RR^N$ satisfying $b=A\bar{z}$, given
a vector $b\in\RR^k$ and a matrix $A \in \RR^{k\times N}$
with significantly fewer rows than columns $(k\ll N)$.
The desired sparse vector~$\bar{z}\in\RR^N$ can be recovered by minimizing the number of non-zero components
by solving
\begin{equation}\label{minl0fourier}
\min \card(z) \quad \operatorname{s.t.}\quad A z= b,
\end{equation}
where $\card(z) = | \{ i \in \{1,\ldots,n \}: z_i \neq 0 \}|$.
Since this problem is generally NP-Hard, one considers a more tractable approximation
by substituting its objective function by a relatively close convex one:
\begin{equation}\label{minl1}
\min \|z\|_1 \quad \operatorname{s.t.}\quad Az=b,
\end{equation}
which is a linear program.
The main results of compressed sensing show that, under certain conditions on the (possibly random) matrix $A$, the solution of~(\ref{minl1}) is in fact $\bar{z}$ and coincides with the optimal solution of~(\ref{minl0fourier}) (possibly, with high probability).
We will now discuss the compressed sensing results that are useful for our purposes.

\subsection{General concepts and properties}\label{sec:CS}

One says that a vector $z$ is $s-$sparse if $\card(z) \leq s$. In compressed sensing, one is interested in matrices $A$ such that, for every $s-$sparse vector $\bar z$, the information given by $b=A\bar{z}$ is sufficient to recover $\bar z$ and, moreover, that such recovery can be accomplished by solving problem~(\ref{minl1}).
The following definition of the \emph{Restricted Isometry Property} (RIP) is introduced in~\cite{ECandes_TTao_2006}.

\begin{definition}[Restricted Isometry Property]\label{defRIP}
One says that $\delta_s>0$ is the Restricted Isometry Property Constant, or \emph{RIP} constant, of order $s$ of the matrix
$A\in\RR^{k\times N}$ if $\delta_s$ is the smallest positive real such that:
\[\left(1-\delta_s\right)\|z\|_2^2 \; \leq \; \|A z\|_2^2 \; \leq \; \left(1+\delta_s\right)\|z\|_2^2\]
for every $s-$sparse vector $z$.
\end{definition}

The following theorem (see, e.g.,~\cite{EJCandes_2009,HRauhut_2010}) provides a useful sufficient condition for successful recovery by~(\ref{minl1}) with $b=A \bar z$.

\begin{theorem}\label{teoremaRIP}
Let $A\in\RR^{k\times N}$ and $2s < N$. If
$\delta_{2s}<\frac1{3}$,
where $\delta_{2s}$ is the RIP constant of $A$ of order~$2s$, then, for every $s-$sparse vector $\bar z$, problem~(\ref{minl1}) with $b=A \bar z$ has a unique solution and it is given by $\bar z$.
\end{theorem}

Although the RIP provides useful sufficient conditions for sparse recovery, it is a difficult and still open problem to find deterministic matrices
which satisfy such a property when the underlying system is highly underdetermined
(see~\cite{Bandeira_etal_FlatRIP,TTao_2007} for a discussion on this topic). It turns out that random matrices provide a better ground for this analysis (see for instance,  one of the
 results in~\cite{RBaraniuk_MDavenport_RDeVore_MWakin_2008}). 
 
\subsection{Partially sparse recovery}\label{sec:partialCS}

To be able to apply sparse recovery results of compressed sensing to our setting we first observe that problem~(\ref{minl1cap2}) is similar to problem~(\ref{minl1}), however, it differs in that only a part of the solution vector $\alpha$ is expected to be sparse and appears in the objective function.
 We hence need to consider an extended recovery result for partial sparsity.

Formally, one has $z=(z_1,z_2)$, where $z_1\in\RR^{N-r}$ is $(s-r)-$sparse and $z_2\in\RR^r$.
A natural generalization of problem~(\ref{minl1}) to this setting of partially sparse recovery is given by
\begin{equation}\label{minpartiall1}
\min \|z_1\|_1 \quad \operatorname{s.t.} \quad A_1 z_1 + A_2 z_2 \; = \; b,
\end{equation}
where $A=(A_1,A_2)$ and $A_1$ has the first $N-r$ columns of $A$ and
$A_2$ the last $r$. One can easily see that problem~(\ref{minl1cap2}) fits into this
formulation by setting  $z_1 = \alpha_Q$, $z_2 = \alpha_L$, $A_1 = M(\bar{\phi}_Q,\Pon)$, $A_2 = M(\bar{\phi}_L,\Pon)$, and $r=n+1$.

We can define an extension of the RIP to the partially sparse recovery setting.
Under the assumption that $A_2$ is full column rank (which in turn is implied by the RIP;
see~\cite{ABandeira_KScheinberg_LNVicente_2011-b}), let
\begin{equation}\label{projM}
{\cal P} \; = \; I-A_2\left(A_2^\top A_2\right)^{-1}A_2^\top
\end{equation}
be the matrix representing the projection from $\RR^N$ onto $\RRR\left(A_2\right)^\bot.$ Then, the problem of recovering
$(\bar{z}_1,\bar{z}_2)$, where $\bar{z}_1$ is an $(s-r)$--sparse vector satisfying $A_1 \bar{z}_1+A_2\bar{z}_2=b$, can be
stated as the problem of recovering an $(s-r)-$sparse vector $\bar{z}_1$ satisfying $\left({\cal P}A_1\right)z_1={\cal P}b$
and then recovering~$\bar{z}_2$ satisfying $A_2z_2=b-A_1 \bar{z}_1$. The latter task results in solving a linear system given that $A_2$ has full column rank and $({\cal P}A_1) \bar{z}_1={\cal P}b$. Note that the former task reduces to the classical setting of compressed sensing. These considerations motivate the following definition of RIP for partially sparse recovery.

\begin{definition}[Partial RIP Property]\label{defpartialRIP}
We say that $\delta_{s-r}^r>0$ is the Partial Restricted Isometry Property Constant of order $s-r$ for recovery of size $N-r$ of the matrix $A=(A_1, A_2)\in\RR^{k\times N}$ (with $A_1 \in \RR^{k\times (N-r)}$,
$A_2 \in \RR^{k \times r}$, and
$r\leq s$) if $A_2$ is full column rank and
$\delta_{s-r}^r$ is the RIP constant of order $s-r$ (see Definition~\ref{defRIP}) of the matrix ${\cal P}A_1$, where ${\cal P}$ is given by~(\ref{projM}).
\end{definition}

When $r=0$ the Partial RIP reduces to the RIP of Definition~\ref{defRIP}.
In~\cite{ABandeira_KScheinberg_LNVicente_2011-b} we show a simple proof
of the fact that if a matrix $A$
 satisfies RIP for $s-$sparse recovery with $\delta_{s}$ constant, then
it  also satisfies Partial RIP with $\delta_{s-r}^r= \delta_{s}$.
A very similar result has been independently proved
in~\cite{LJacques_2010}.
It is also shown in~\cite{ABandeira_KScheinberg_LNVicente_2011-b}
 that Partial RIP implies that the solution of~(\ref{minpartiall1}) is the
original $s-$sparse solution $\bar{z}=(\bar{z}_1,\bar{z}_2)$.
Hence to be able to apply sparse recovery results to  problem~(\ref{minl1cap2}),
which is of interest to us, it suffices to construct matrices $M(\bar{\phi},\Pon)$
for which the RIP property holds. In~\cite{NVaswani_WLu_2011} a specific sufficient condition for partially
sparse  recovery is given, but it remains to be seen if we can use such a result
to strengthen the bounds on the sample set size which we derive in Section~\ref{chapter4}.
To establish these bounds, we will rely on results on random matrices which
apply to our specific setting. We discuss these results in the next section.

\section{Recovery of Sparse Hessians}\label{chapter4}

\subsection{Sparse recovery using orthonormal bases}\label{sec:41}

For the purposes of building quadratic models based on sparse Hessians
we are interested in solving~(\ref{minl1cap2}) which is equivalent to~(\ref{minpartiall1}), where
$A_1=M(\phi_Q,Y)$, $A_2=M(\phi_L,Y)$,
$z_1 = \coe_Q$, $z_2=\coe_L$, $b=\fo(\Pon)$, and $r=n+1$.
In this case $\phi$ is a basis in the space $\PPP_n^2$ of polynomials of degree $\leq 2$ of
dimension $N=(n+1)(n+2)/2$ and the resulting quadratic model $\qua$ is constructed as
$$
\qua(x) \; = \; \sum_{\iota=1}^N \alpha_\iota \phi_\iota(x)
$$
where $\alpha$ is the vector of coefficients which is presumed to be sparse (with partially known support since~$\alpha_L$
is not necessarily sparse).

Let us now consider a general setting of a finite dimensional space of functions (defined in some domain $\drau$) spanned by a basis  $\phi=\{\phi_1,...,\phi_N\}$ of functions (not necessarily polynomial).
Let us also consider a function $\frau:\drau\to\RR$ which belongs to that space, in other words $\frau$ can be written as $$\frau \; = \; \sum_{j=1}^N\coe_j\phi_j,$$
for some expansion coefficients $\coe_1,...,\coe_N$. We are interested in the problem of recovering $\frau$ from its values in some finite subset $\Pon=\{\pon^1,...,\pon^k\}\subset \drau$ with $k\leq N$, with the additional assumption that $\frau$ is $s-$sparse, meaning that the expansion coefficient vector $\coe$ is $s-$sparse.
The purpose of this section is to provide conditions under which such
 recovery occurs with high probability. Although the results of this
section hold also for complex valued functions, we will restrict ourselves to the real case, because the functions we are interested in DFO are real valued. We consider a probability measure $\mu$ defined in $\drau$ (having in mind that $\drau\subset\RR^n$). The basis $\phi$ will be required to satisfy the following orthogonality property~\cite{HRauhut_2010}.

\begin{definition}[$K$-bounded orthonormal basis]\label{kboundedness}
A set of functions $\phi=\{\phi_1,...,\phi_N\}$, spanning a certain function space, is said to be an orthonormal basis satisfying the $K$-boundedness condition (in the domain $\drau$ for the measure $\mu$) if
$$\int_\drau\phi_i(\vad)\phi_j(\vad)d\mu(\vad) \; = \; \delta_{ij},$$
(here $\delta_{ij}$ is the Kronecker delta) and $\|\phi_j\|_{L^\infty(\drau)}\leq K$, for all $i,j\in \{1,\ldots,N\}$.
\end{definition}

The following theorem (see~\cite[Theorem~4.4]{HRauhut_2010}) shows that by selecting the sample set
$\Pon$ randomly we can recover the sparse coefficient vector
with fewer sample points than basis coefficients.

\begin{theorem} \label{rauhutuniforme}
Let $M(\phi,\Pon)\in\RR^{k\times N}$ be the interpolation matrix associated with an orthonormal basis satisfying the $K$-boundedness condition. Assume that the sample set $\Pon=\{\pon^1,...,\pon^k\}\subset\drau$ is chosen randomly where each point is drawn independently according to the probability measure $\mu$. Further assume that
\begin{eqnarray}
\frac{k}{\log k}&\geq& c_1K^2s(\log s)^2 \log N,\label{kmaiorepsilonantes}\\
k&\geq& c_2K^2s\log\left(\frac1\varepsilon\right),\label{kmaiorepsilon}
\end{eqnarray}
where $c_1,c_2>0$ are universal constants,  $\varepsilon\in(0,1)$, and $s\in \{1, \ldots, N\}$. 
Then, with probability at least $1-\varepsilon$, $\frac1{\sqrt{k}}A=\frac1{\sqrt{k}}M(\phi,\Pon)$  
satisfies the RIP property (Definition~\ref{defRIP}) with constant $\delta_{2s}<\frac1{3}$.
\end{theorem}

From the classical results in compressed sensing (see Theorem~\ref{teoremaRIP} and the paragraph afterwards),
this result implies that every $s-$sparse vector $\bar{z} \in \RR^N$
is the unique solution to the $\ell_1$-minimization problem~(\ref{minl1}),
with $A=M(\phi,\Pon)$ and $b=M(\phi,\Pon) \bar{z} =\frau(\Pon)$.
However, it also implies, by~\cite[Theorem~4.2]{ABandeira_KScheinberg_LNVicente_2011-b},
that every $s-$sparse vector $(\bar{z}_1,\bar{z}_2)$
with $(s-r)-$sparse $\bar{z}_1\in \RR^{N-r}$ and possibly dense $\bar{z}_2\in \RR^r$, is the unique solution to the $\ell_1$-minimization problem~(\ref{minpartiall1}), with $A=M(\phi,\Pon)$ and $b=M(\phi,\Pon) \bar{z} =\frau(\Pon)$. Note that it is a scaled version of $A$,
given by $A/\sqrt{k}$ and not $A$ itself, that satisfies the RIP property but this does not affect the recovery results in the exact setting. As we will see below the scaling
has an effect in the noisy case.

It is worth noting that an optimal result is obtained if one sets
$\varepsilon=e^{-\frac{k}{c_2K^2s}}$ in the sense that~(\ref{kmaiorepsilon})
is satisfied with equality. Also, from~(\ref{kmaiorepsilonantes})
we obtain $k\geq (\log k) c_1K^2$ $s(\log s)^2 \log N$,
and so, using $\log s\geq 1$,
$1-e^{-\frac{k}{c_2K^2s}}\geq 1-N^{-\gamma\log k}$, for the universal
constant $\gamma=c_1/c_2$. Thus, $\varepsilon$ can be set such that the
probability of success $1-\varepsilon$ satisfies
\begin{equation}\label{epsilonoptimoK}
1-\varepsilon \; \geq \; 1-N^{-\gamma\log k},
\end{equation}
showing that this probability
grows polynomially with $N$ and $k$.

As we observe later in this section, we are not  interested in satisfying
the interpolation conditions exactly, hence we need to consider
instead of
$b=\frau(\Pon)$ a perturbed  version $b=\frau(\Pon)+\epsilon$, with a known bound on the size of $\epsilon$. In order to extend the results we just described to the case of noisy recovery, some modifications of problem~(\ref{minl1})
are needed. In the case of full  noisy recovery  it is typical to consider, instead of the formulation~(\ref{minl1}), the following optimization problem:
\begin{equation}\label{minl1noisy}
\min \|z\|_1 \quad \operatorname{s.t.}\quad \|Az-b\|_2 \; \leq \; \eta,
\end{equation}
where $\eta$ is a positive number. We now present a recovery result based on the formulation~(\ref{minl1noisy}) and thus appropriate to the noisy case. A proof is available in~\cite{EJCandes_2009}.

\begin{theorem}\label{rauhutnoisy}
Under the same assumptions of Theorem~\ref{rauhutuniforme}, with probability at least $1-\varepsilon$, $\varepsilon\in(0,1)$, the following holds for every $s-$sparse vector $\bar z$:

Let noisy samples $b=M(\phi,\Pon) \bar{z}+\epsilon$ with
$$\|\epsilon\|_2 \; \leq \; \eta$$
be given, for any $\eta$ non-negative, and let $z^\ast$ be the solution of the $\ell_1$-minimization problem~(\ref{minl1noisy}) with $A=M(\phi,\Pon)$. Then,
\begin{equation}\label{herefortheremark}
 \|z^\ast - \bar{z}\|_2 \; \leq \; \frac{c_{total}}{\sqrt{k}}\,\eta
\end{equation}
for some universal constant $c_{total}>0$.
\end{theorem}

Since we are interested in the partially sparse recovery case, we need to consider instead
\begin{equation}\label{minl1noisy-partial}
\min \|z_1\|_1 \quad \operatorname{s.t.}\quad \|Az-b\|_2 \; \leq \; \eta.
\end{equation}
The extension of Theorem~\ref{rauhutnoisy} to partially recovery for the noisy case is obtained from the
full noisy recovery, analogously to the exact case (see~\cite[Theorem~5.2]{ABandeira_KScheinberg_LNVicente_2011-b}
for a proof).

\begin{theorem} \label{th:noisy_partial}
Under the same assumptions of Theorem~\ref{rauhutuniforme}, with probability at least $1-\varepsilon$, $\varepsilon\in(0,1)$, the following holds for every vector $\bar z = ({\bar z}_1, {\bar z}_2 )$ with $r \leq s$ and $\bar z_1$ an $(s-r)-$sparse vector:

Let noisy samples $b=M(\phi,\Pon) \bar{z}+\epsilon$ with
$$\|\epsilon\|_2 \; \leq \; \eta$$
be given, for any~$\eta$ non-negative, and let $z^\ast=(z^*_1,z^*_2 )$ be the solution of the $\ell_1$-minimization problem~(\ref{minl1noisy-partial}) with $A=M(\phi,\Pon)$. Then,
\begin{equation}
\|z^\ast - \bar{z}\|_2 \leq \frac{c_{partial}}{\sqrt{k}}\,\eta, \label{noisy-bound-partial}
\end{equation}
for some universal constant $c_{partial}>0$.
\end{theorem}

%It is worth noting that the factor $1/\sqrt{k}$ appears
%in~(\ref{herefortheremark}) and (\ref{noisy-bound-partial}) due to the fact that it is $A/\sqrt{k}$
%(and not $A$, necessarily) that
%satisfies the RIP property.
%Also, we
Note that it is possible
to extend these results to  approximately sparse
vectors (see~\cite{ABandeira_KScheinberg_LNVicente_2011-b}), however we
do not  include such an extension in the present paper for the sake of clarity of
the exposition.

\subsection{Sparse recovery using polynomial orthonormal expansions}\label{seccaopolort}

As described in Section~\ref{chapter2}, we are  interested in
 recovering a local quadratic model of the objective function $\fo:D\subset\RR^n\to\RR$ near a point $x_0$. Therefore we consider  the space of quadratic functions defined in $B_p(x_0;\Delta)$. To apply the results in Theorems
\ref{rauhutuniforme} and \ref{rauhutnoisy} we need to build an appropriate orthonormal basis for the space of quadratic functions in $B_p(x_0;\Delta)$. In addition we
 require  that the models we recover are expected to be sparse in such a
 basis. In this paper we consider  models that reconstruct sparse Hessians
of  $\fo$, and thus it is natural to include into the basis polynomials of the forms
$c_{ij}\left(\vad_i-x_0\right)\left(\vad_j-x_0\right)$, with some constant
$c_{ij}$ (we will henceforth set $x_0=0$ in this section without lost of generality). It is then required that these elements of the basis do not  appear as parts of other basis polynomials.  The orthonormal basis should satisfy the $K$-boundedness condition for some constant $K$ independent\footnote{Otherwise the results in Theorems~\ref{rauhutuniforme} and~\ref{rauhutnoisy} become weaker. Recently, progress has been made in addressing the  case  when $K$ grows with the dimension, where the main idea is to precondition the interpolation matrix (see~\cite{HRauhut_RWard_2010}).}  of the dimension $n$.

We will now build such an orthonormal basis on the domain $\drau=B_\infty(0;\Delta)=[-\Delta,\Delta]^n$
(the $\ell_\infty$-ball centered at the origin and of radius $\Delta$),
using the uniform probability measure $\mu$ and the corresponding $L^2$ inner product.

\subsubsection{An orthonormal basis on hypercubes}\label{subsec:HC}

Let $\mu$ be the uniform probability measure on $B_\infty(0;\Delta)$. Note that due to the geometric properties of $B_\infty(0;\Delta)=[-\Delta,\Delta]^n$, one has
\begin{equation}\label{produtointegrais}\int_{[-\Delta,\Delta]^n}g(\vad_i)h(\vad_1,...,\vad_{i-1},\vad_{i+1},...,\vad_n)d\vad \; =\end{equation}
\begin{equation*}
= \; \int_{-\Delta}^\Delta g(\vad_i)d\vad_i\int_{[-\Delta,\Delta]^{n-1}}h(\vad_1,...,\vad_{i-1},\vad_{i+1},...,\vad_n)d\vad_1\cdots d\vad_{i-1}d\vad_{i+1}\cdots d\vad_n,
\end{equation*}
for appropriate functions $g$ and $h$ satisfying the conditions of Fubini's Theorem.

We want to find an orthonormal basis, with respect to $\mu$, of the second degree polynomials on $B_\infty(0;\Delta)$ that contains  the polynomials $\{c_{ij}\vad_i\vad_j\}_{i\neq j}$. We are considering, first, the off-diagonal part of the
 Hessian since this is the part which is expected to be sparse (indicating the lack of direct variable interactions). It is easy to see that  the $n(n-1)/2$ polynomial functions $\{c_{ij}\vad_i\vad_j\}_{i\neq j}$ are all orthogonal, and that
due to symmetry all $c_{ij}$ constants are equal, to say, $k_2$ (a normalizing constant). Hence we have $n(n-1)/2$ elements of the basis. Now, note that from~(\ref{produtointegrais}), for different indices $i,j,l$,
\[
\int_{B_\infty(0;\Delta)}\vad_i\vad_j\vad_ld\mu \; = \; \int_{B_\infty(0;\Delta)}\vad_i\vad_jd\mu \; = \; \int_{B_\infty(0;\Delta)}\vad_i\vad_j^2d\mu \; = \;0.
\]
As a result, we can add to the set $\{k_2\vad_i\vad_j\}_{i\neq j}$ the polynomials $\{k_1\vad_i\}_{1\leq i\leq n}$ and the polynomial~$k_0$, where $k_1$ and $k_0$ are normalizing constants, forming a set of $n(n-1)/2+(n+1)$ orthogonal polynomials.

It remains to construct $n$ quadratic polynomials, which have to contain terms
$\vad_i^2$ but should not contain terms $\vad_i\vad_j$. We choose to consider
 $n$ terms of the form $k_3(\vad_i^2-\alpha_1\vad_i-\alpha_0)$. We will select
 the constants $\alpha_0$ and $\alpha_1$ in such a way that these polynomials are orthogonal to the already constructed ones.
From the  orthogonality with respect to $k_i\vad_i$, i.e.,
$$\int_{B_\infty(0;\Delta)}\vad_i(\vad_i^2-\alpha_1\vad_i-\alpha_0)d\mu \; = \; 0,$$
we must have $\alpha_1=0$. Then, orthogonality with respect to the constant
 polynomial $k_0$ implies
$$\int_{B_\infty(0;\Delta)}(\vad_i^2-\alpha_0)d\mu \; = \; 0.$$
Thus,
$$\alpha_0 \; = \; \frac1{2\Delta}\int_{-\Delta}^\Delta \vad^2d\vad \; = \;
\frac1{2\Delta}\left(\frac23\Delta^3\right) \; = \; \frac13\Delta^2.$$

Hence we have a set of orthogonal polynomials that span the set of quadratic functions on $B_\infty(0;\Delta)$. What remains is  the computation of the normalization constants to ensure normality of basis elements. From
\begin{equation*}
\int_{B_\infty(0;\Delta)}k_0^2d\mu \; = \; 1
\end{equation*}
we set $k_0=1$. From the equivalent statements
\begin{eqnarray*}
\int_{B_\infty(0;\Delta)}\left(k_1\vad_i\right)^2d\mu \; = \; 1,\\
\frac{k_1^2}{(2\Delta)^n}\int_{-\Delta}^\Delta\vad^2d\vad\int_{[-\Delta,\Delta]^{n-1}}1d\vad \; = \; 1,\\
k_1^2\int_{-\Delta}^\Delta \vad^2\frac{d\vad}{2\Delta} \; = \; 1,
\end{eqnarray*}
we obtain $k_1=\sqrt{3}/\Delta$. From the equivalent statements
\begin{eqnarray*}
\int_{B_\infty(0;\Delta)}\left(k_2\vad_i\vad_j\right)^2d\mu \; = \; 1,\\
k_2^2\left(\int_{-\Delta}^\Delta \vad^2\frac{d\vad}{2\Delta}\right)^2 \; = \; 1,
\end{eqnarray*}
we conclude that $k_2=3/\Delta^2$. And from the equivalent statements
\begin{eqnarray*}
\int_{B_\infty(0;\Delta)}\left(k_3\left(\vad_i^2-\frac13\Delta^2\right)\right)^2d\mu \; = \; 1,\\
k_3^2\int_{-\Delta}^\Delta\left(\vad^2-\frac13\Delta^2\right)^2\frac{1}{2\Delta}d\vad \; = \; 1,
\end{eqnarray*}
we obtain
\[
k_3 \; = \; \frac{3\sqrt{5}}{2\Delta^2}.
\]
We have thus constructed the desirable basis, which we will denote by $\psi$.
We will  abuse the notation to define  $\psi$  using indices $(0),(1,i)$, $(2,ij)$ or $(2,i)$ for the elements of $\psi$ in place  of
the single index $\iota$. The expressions of these sophisticated indices should simplify the understanding. For instance, $(2,ij)$ index stands for the  element of the basis $\psi$ which involves the term $\vad_i\vad_j$, similarly $\alpha_{(2,i)}$  is the term corresponding to $\vad_i^2$,
and so on.

\begin{definition}\label{defbasehipercubo}
We define the basis $\psi$ as the set of the following $\dimquah$ polynomials:
\begin{equation}\label{basehipercubo}\left\{\begin{array}{ccl}
\psi_{2,i}(\vad)&=&\frac{3\sqrt{5}}{2\Delta^2}\vad_i^2-\frac{\sqrt{5}}2, \\
\psi_{2,ij}(\vad)&=&\frac3{\Delta^2}\vad_i\vad_j, \\
\psi_{1,i}(\vad)&=&\frac{\sqrt{3}}{\Delta}\vad_i, \\
\psi_{0}(\vad)&=&1. \\
\end{array}\right.\end{equation}\end{definition}

The basis $\psi$ satisfies the assumptions of Theorems~\ref{rauhutuniforme} and~\ref{rauhutnoisy}, as stated in the following theorem.
\begin{theorem}\label{teoremabasehipercubo}
The basis $\psi$ (see Definition~\ref{defbasehipercubo}) is orthonormal and satisfies the $K$-boundedness condition (see Definition~\ref{kboundedness}) in $B_\infty(0;\Delta)$ for the uniform probability measure with $K=3$.
\end{theorem}

\begin{proof}
From the above derivation and~(\ref{produtointegrais}) one can easily show that $\psi$ is orthonormal in $B_\infty(0;\Delta)$ with respect to the uniform probability measure.
So, it remains to prove the boundedness condition with $K=3$. In fact, it is easy to check that
\begin{equation}\label{basehipercubobound}
\left\{\begin{array}{cccc}
\|\psi_{2,i}\|_{L^\infty(B_\infty(0;\Delta))}&=&\sqrt{5}&\leq3, \\
\|\psi_{2,ij}\|_{L^\infty(B_\infty(0;\Delta))}&=&3&\leq3, \\
\|\psi_{1,i}\|_{L^\infty(B_\infty(0;\Delta))}&=&\sqrt{3}&\leq3, \\
\|\psi_{0}\|_{L^\infty(B_\infty(0;\Delta))}&=&1&\leq3, \\
\end{array}\right .
\end{equation}
where $\|g\|_{L^\infty(B_\infty(0;\Delta))}=\max_{x\in B_\infty(0;\Delta)}|g(x)|$.
\end{proof}

We will consider $\psi_Q$, the subset of $\psi$ consisting of the
polynomials of degree $2$, and $\psi_L$, the ones of degree $1$ or $0$,
as we did in Section~\ref{chapter2} for $\bar{\phi}$.

We are interested in quadratic functions $\qua=\sum_{\iota}\coe_\iota\psi_\iota$ (see Definition~\ref{defbasehipercubo}) with an $h-$sparse  coefficient subvector
$\coe_Q$, i.e., only $h$ coefficients corresponding to the polynomials in~$\psi_Q$ in the representation of $\qua$ are non-zero, where $h$ is a number between $1$ and $n(n+1)/2$. In such cases, the corresponding full vector~$\coe$ of coefficients is $(h+n+1)-$sparse. We now state a corollary of Theorem~\ref{rauhutnoisy} for sparse recovery in the orthonormal basis $\psi$, with $k=p$ (number of sample points) and $N=q$ (number of elements in $\psi$), which will  be used in the next section to establish results on sparse quadratic model recovery. Note that we write the probability of successful recovery of a sparse solution
 in the form $1-n^{-\gamma\log p}$ which can be derived from~(\ref{epsilonoptimoK}) using $q=\OOO(n^2)$ and a simple modification of the universal constant $\gamma$.

\begin{corollary}\label{rauhutnoisyHC}
Let $M(\psi,\Pon)\in\RR^{p\times q}$ be the matrix
with  entries
$\left[M(\psi,\Pon)\right]_{ij}=\psi_j(\pon^i)$, $i=1,...,p$, $j=1,...,q$, with $q=\dimquah$.

Assume that the sample set $\Pon=\{\pon^1,...,\pon^p\}\subset B_\infty(0;\Delta)$ is chosen randomly where each point is drawn independently according to the
uniform probability  measure $\mu$ in $B_\infty(0;\Delta)$. Further assume that
$$\frac{p}{\log p} \; \geq \; 9c\left(h+n+1\right)\left(\log \left(h+n+1\right)\right)^2 \log q,$$
for some universal constant $c>0$ and $h\in\{1,...,n(n+1)/2\}$.
Then, with probability at least $1-n^{-\gamma\log p}$, for some universal constant $\gamma>0$, the following
holds for every vector $\bar z$, having at most $h+n+1$ non-zero expansion coefficients in the basis $\psi$:

Let noisy samples $b=M(\psi,\Pon) \bar{z}+\epsilon$ with
$$\|\epsilon\|_2 \; \leq \; \eta$$
be given, for any $\eta$ non-negative, and let $z^\ast$ be the solution of the $\ell_1$-minimization problem~(\ref{minl1noisy-partial}) with $A=M(\psi,\Pon) = ( M(\psi_Q,\Pon), M(\psi_L,\Pon) )$. Then,
$$\|z^\ast - \bar{z}\|_2 \; \leq \; \frac{c_{partial}}{\sqrt{p}}\,\eta$$
for some universal constant $c_{partial}>0$.
\end{corollary}

\begin{remark}
It would be natural to consider the interpolation domain to be
the  ball $B_2(0;\Delta)$ in the classical $\ell_2$-norm.  However, our procedure of constructing an orthonormal set of polynomials with desired properties in the  hypercube, i.e., using the $\ell_\infty$-norm ball, does not extend naturally to the $\ell_2$ one.  One problem with the uniform measure in the $\ell_2$-ball is that
 formulas like~(\ref{produtointegrais}) no longer hold. A construction of an appropriate basis for the uniform measure on the $\ell_2$-ball is a subject for further work.
\end{remark}

%***************************************************************************************************************************************************************
\subsection{Recovery of a fully quadratic model of a function with sparse Hessian}\label{sec:43}

As we stated earlier our main interest in this paper is to recover
a fully quadratic model (see Definition~\ref{deffullyquadmodel})
of a twice continuously differentiable objective function
$\fo:D\to \RR$ near a point~$x_0$ using fewer than $(n+1)(n+2)/2$ sample points.
In other words, we want to show that for  a given function $f$ there exist
constants $\kappa_{ef}$, $\kappa_{eg}$, and $\kappa_{eh}$
such that, given any point $x_0$ and a radius~$\Delta$, we can build a model
based on a
random sample set of $p$ points (with $p< (n+1)(n+2)/2$) which, with high probability, is a fully quadratic model of $f$ on $B_p(x_0;\Delta)$ with respect to
the given constants $\kappa_{ef}$, $\kappa_{eg}$, and $\kappa_{eh}$. The number $p$ of sample points depends on the sparsity of the Hessian of the model that we are attempting to reconstruct.
Hence we need to make some assumption about the sparsity.
The simplest (and strongest) assumption we can make is that the function $f$ has a sparse Hessian at any point $x_0$.

\begin{assumption}[Hessian sparsity]\label{assumptionhessiansparse}
Assume that $\fo:D\to \RR$ satisfies Assumption~\ref{assumptionquadratica} and furthermore that for any given
$x_0\in D$ the Hessian $\nabla^2f(x_0)$ of $\fo$ at $x_0$ has at most $h$ non-zero entries, on or above the diagonal,
where $h$ is a number between $1$ and $n(n+1)/2$. If this is the case, then $\nabla^2f $  is said to be $h-$sparse.
\end{assumption}

The above assumption implies that for every $x_0\in D$ there exists a fully quadratic  second degree polynomial model $q_f$ of $\fo$ such that the Hessian $\nabla^2q_f$ is $h-$sparse,  from a fully quadratic class with $\kappa^\prime_{ef}$,  $\kappa_{eg}^\prime$, and $\kappa^\prime_{eh}$
equal to some multiples of the Lipschitz constant of~$\nabla^2f$.
The second order Taylor model at $x_0$ is, in particular, such a model.

However, we do not need this strong assumption to be able to construct fully quadratic models. Constructing models via random sample sets and $\ell_1$-minimization, in the way that we described above, provides fully quadratic models regardless of the amount of sparsity of the Hessian, as we will show in this section.
 The sparsity of the Hessian affects, however, the number of sample points that
are required.  Hence, one can consider functions whose Hessian is approximately sparse and the sparsity pattern, or even the cardinality
(number of non-zeros), is not constant. The following assumption is weaker than Assumption~\ref{assumptionhessiansparse} but is sufficient for our purposes.

\begin{assumption}[Approximate Hessian sparsity]\label{assumptionhessiansparse-app}
Assume that $\fo:D\to \RR$ satisfies Assumption~\ref{assumptionquadratica} and furthermore that for any given $x_0\in D$ and $\Delta >0$
 there exists a second degree polynomial
 $\qua(\vad)=\sum_{\iota}\alpha_\iota \psi_\iota(\vad)=\alpha_Q\psi_Q(\vad)+\alpha_L\psi_L(\vad)$, with $\alpha_Q$ an $h-$sparse coefficient vector, where $h$ may depend on $x_0$ and $\Delta$, which is a fully quadratic  model  of $\fo$ on $B_p(x_0;\Delta)$ for some constants $\kappa^\prime_{ef}$, $\kappa_{eg}^\prime$, and $\kappa^\prime_{eh}$, independent of $x_0$ and $\Delta$.
\end{assumption}

If in the above assumption $h$ is independent of $x_0$ and $\Delta$,
then the assumption reduces to Assumption~\ref{assumptionhessiansparse}.
As it stands, Assumption~\ref{assumptionhessiansparse-app} is less restrictive.

Given the result in Section~\ref{seccaopolort}, we will consider the $\ell_\infty$-norm in Definition~\ref{deffullyquadmodel}, thus considering regions of the form $B_\infty(x_0;\Delta)$.

When we state that $\fo$ has a sparse Hessian, it is understood that
the representation of  the Taylor second order expansion, or of any other
fully quadratic model of $\fo$, is a sparse linear combination of the elements of the canonical basis $\bar{\phi}$ (see~(\ref{basecanonicaigualdade})).
However, the basis~$\bar{\phi}$ is not  orthogonal on $B_\infty(x_0;\Delta)$. Hence we are interested in models that have a sparse representation in the orthonormal basis~$\psi$ of Definition~\ref{defbasehipercubo}.
Fortunately,  basis $\psi$ can be obtained from $\bar\phi$   through a few simple transformations. In particular, the sparsity  of the Hessian of a quadratic model $\qua$ will be carried over to sparsity in the representation of $\qua$ in $\psi$, since the expansion coefficients in $\psi_Q$ will be multiples of the ones in~$\bar{\phi}_Q$, thus guaranteeing that if the coefficients in the latter are $h-$sparse, so are the ones in the former.

We are now able to use the material developed in Section~\ref{seccaopolort} to
guarantee, with high probability, the construction, for each $x_0$ and $\Delta$, of a fully quadratic model of $\fo$ in $B_\infty(x_0;\Delta)$ using a
random sample set of only $\OOO(n(\log n)^4)$ points, instead of  $\OOO(n^2)$ points, provided that
$h = \OOO(n)$ (for the given $x_0$ and $\Delta$, see Assumption~\ref{assumptionhessiansparse-app}).
We find such a fully quadratic model by solving the partially
sparse recovery version of problem~(\ref{minl1noisy})
 written now in the form
\begin{equation}\label{minl1HC0}\begin{array}{cl}
\min &\|\coe_Q\|_1\\ [1ex]
\operatorname{s.t.} &\left\|M(\psi_Q,\Pon)\coe_Q+M(\psi_L,\Pon)\coe_L-\fo(\Pon + x_0)\right\|_2 \; \leq \; \eta, \end{array}
\end{equation}
where $\eta$ is some appropriate positive quantity and $Y$ is drawn in $B_\infty(0;\Delta)$.
Corollary~\ref{rauhutnoisyHC} can then be used to ensure  that only $\OOO(n(\log n)^4)$ points are
necessary for recovery of a sparse model in $B_\infty(0;\Delta)$, when the number of non-zero
components of the Hessian of~$m$ in Assumption~\ref{assumptionhessiansparse-app} is of the order of $n$.

Note that we are in fact considering ``noisy'' measurements, because we are only able to evaluate the function $\fo$ while trying to recover a fully quadratic model, whose values are somewhat different from those of $\fo$.
We will say that a function $q^\ast$ is the solution to the minimization problem~(\ref{minl1HC0}) if $q^\ast(\vad)=\sum_{\iota}\alpha^\ast_l\psi_\iota(\vad)$, where $\alpha^\ast$ is the minimizer of~(\ref{minl1HC0}).

First we need to prove an auxiliary lemma.
Corollary~\ref{rauhutnoisyHC}  provides an estimate on the $\ell_2$-norm of the error in the recovered
 vector of coefficients
of the quadratic model. In the definition of fully quadratic models, the error between the quadratic model and the function $\fo$
is measured in terms of the maximum difference of
 their function values in $B_\infty(x_0;\Delta)$ and the maximum norms of the differences of their gradients  and their Hessians in
 $B_\infty(x_0;\Delta)$. The following lemma establishes a bound for the value, gradient, and Hessian of quadratic polynomials in $B_\infty(0;\Delta)$
 in terms of the norm of their coefficient vector (using the basis~$\psi$).

\begin{lemma}\label{quadraticL2errors}
Let $\qua$ be a quadratic function and  $\alpha$ be a vector in $\RR^{\dimquah}$ such that
$$\qua(\vad) \; = \; \sum_{\iota}\alpha_\iota\psi_\iota(\vad)$$  with $\psi(x)$ defined in~(\ref{basehipercubo}).
Then
\begin{eqnarray*}
\left|\qua(\vad)\right|&\leq&\left(3\sqrt{\card(\alpha)}\right)\|\alpha\|_2 \label{quadraticL2Bound0}\\
\left\|\nabla\qua(\vad)\right\|_2&\leq&\left(3\sqrt{5}\sqrt{\card(\alpha)}\right)\frac1\Delta\|\alpha\|_2\label{quadraticL2Bound1}\\
\left\|\nabla^2\qua(\vad)\right\|_2&\leq&\left(3\sqrt{5}\sqrt{\card(\alpha)}\right)\frac1{\Delta^2}\|\alpha\|_2,\label{quadraticL2Bound2}
\end{eqnarray*}
for all $\vad\in B_\infty(0;\Delta)$, where $\card(\alpha)$ is the number of non-zero elements in $\alpha$.
\end{lemma}

\begin{proof}
We will again use the  indices $(0),(1,i)$, $(2,ij)$ or $(2,i)$ for the elements of $\alpha$ in correspondence to the indices used in Definition \ref{defbasehipercubo}.

From the $K$-boundedness conditions~(\ref{basehipercubobound}) we have
$$|\qua(\vad)| \; \leq \; \sum_{\iota}|\alpha_\iota| |\psi_\iota(\vad)| \; \leq \; 3\|\alpha\|_1
\; \leq \; 3\sqrt{\card(\alpha)}\|\alpha\|_2,$$
for all $\vad\in B_\infty(0;\Delta)$.
Also, from~(\ref{basehipercubo}),
\begin{eqnarray*}
\left| \frac{\partial\qua}{\partial \vad_i}(\vad) \right| &\leq&\sum_{\iota}|\alpha_\iota|\left |\frac{\partial\psi_\iota}{\partial \vad_i}\right |\\
&=&|\alpha_{1,i}|\left |\frac{\sqrt{3}}{\Delta}\right |+\sum_{j\in \{1,\ldots,n\} \setminus\{i\}}|\alpha_{2,ij}|\left |\frac3{\Delta^2}\vad_j\right | +\, |\alpha_{2,i}|\left |\frac{3\sqrt{5}}{\Delta^2}\vad_i\right |\\
&\leq &\frac{\sqrt{3}}{\Delta}|\alpha_{1,i}|+\sum_{j\in \{1,\ldots,n\} \setminus\{i\}}\frac3{\Delta}|\alpha_{2,ij}|+\frac{3\sqrt{5}}{\Delta}|\alpha_{2,i}|.
\end{eqnarray*}
Then, by the known relations between the norms $\ell_1$ and $\ell_2$,
\begin{eqnarray*}
\left\|\nabla\qua(\vad)\right\|_2 \; \leq \; \left\|\nabla\qua(\vad)\right\|_1 & \leq &
\sum_{i=1}^n \left( \frac{\sqrt{3}}{\Delta}|\alpha_{1,i}|+\sum_{j\in \{1,\ldots,n\}
\setminus\{i\}}\frac3{\Delta}|\alpha_{2,ij}|+\frac{3\sqrt{5}}{\Delta}|\alpha_{2,i}| \right) \\
& \leq & \sum_{i=1}^n \frac{\sqrt{3}}{\Delta}|\alpha_{1,i}|+\sum_{i,j \in \{1,\ldots,n\},
j > i}\frac6{\Delta}|\alpha_{2,ij}|+ \sum_{i=1}^n \frac{3\sqrt{5}}{\Delta}|\alpha_{2,i}| \\
& \leq & \left(3\sqrt{5}\sqrt{\card(\alpha)}\right)\frac1{\Delta}\|\alpha\|_1,
\end{eqnarray*}
for all $\vad\in B_\infty(0;\Delta)$.

For the estimation of the Hessian, we need to separate the diagonal from the non-diagonal part. For the non-diagonal part, with $i\neq j$,
\[
\left| \frac{\partial^2\qua}{\partial \vad_i\partial \vad_j}(\vad) \right|
\; \leq \; \sum_{\iota}|\alpha_\iota|\frac{\partial^2\psi_\iota(\vad)}{\partial \vad_i\partial \vad_j}
\; \leq \; |\alpha_{2,ij}|\frac3{\Delta^2}.
\]
For the diagonal part, with $i=1,\ldots, n$,
\[
\left|
\frac{\partial^2\qua}{\partial \vad_i^2}(\vad) \right| \; \leq \; \sum_{\iota}|\alpha_\iota|
\frac{\partial^2\psi_\iota(\vad)}{\partial \vad_i^2}
\; \leq \; |\alpha_{2,i}|\frac{3\sqrt{5}}{\Delta^2}.
\]
Since the upper triangular part of the Hessian has at most $\card(\alpha)$ non-zero components one has
\begin{eqnarray*}
\left\|\nabla^2\qua(\vad)\right\|_F
& = &
\sqrt{ \sum_{i,j \in \{1,\ldots,n\}} \left( \frac{\partial^2\qua}{\partial \vad_i\partial \vad_j}(\vad) \right)^2}
\\
& \leq & \sum_{i,j \in \{1,\ldots,n\}} \left| \frac{\partial^2\qua}{\partial \vad_i\partial \vad_j}(\vad) \right| \\
& \leq & \sum_{i,j \in \{1,\ldots,n\}, j>i} |\alpha_{2,ij}|\frac6{\Delta^2} +
\sum_{i \in \{1,\ldots,n\}} |\alpha_{2,i}|\frac{3\sqrt{5}}{\Delta^2} \\
& \leq & \sqrt{\card(\alpha)} \left( \frac{3\sqrt{5}}{\Delta^2} \right) \|\alpha\|_2.
\end{eqnarray*}
Thus,
$$ \left\|\nabla^2\qua(\vad)\right\|_2 \; \leq \;
\left\|\nabla^2\qua(\vad)\right\|_F \; \leq \; \sqrt{\card(\alpha)}
\left( \frac{ 3\sqrt{5} }{\Delta^{2} } \right) \|\alpha\|_2$$
for all $\vad\in B_\infty(0;\Delta)$.
\end{proof}

\begin{remark}
The dependency of the error bounds  in Lemma~\ref{quadraticL2errors}   on ${\card(\alpha)}$ cannot be eliminated. In fact, the quadratic function
$$\frau(\vad) \; = \; \sum_{i,j \in \{1,\ldots,n\}, j>i } \sqrt{\frac{2}{n(n-1)}}\frac3{\Delta^2}\vad_i\vad_j$$
satisfies  $\frau(\Delta,...,\Delta)= 3 \sqrt{(n(n-1))/2} = 3 \sqrt{\card(\alpha)}$ while the vector of coefficients~$\alpha$ has norm equal to $1$.
\end{remark}

\begin{remark}
 Since $\psi$ is orthonormal (with respect to $\mu$) on $B_\infty(0;\Delta)$
 we have that
$\|\alpha\|_2=\|\qua\|_{L^2(B_\infty(0;\Delta),\mu)}$. Hence the $\ell_2$-norm of the vector of the coefficients is simply the $L^2$-norm of the function $\qua$ over $B_\infty(0;\Delta)$. If we now consider $\|\qua\|_{L^\infty(B_\infty(0;\Delta))}$, the $L^\infty$-norm of $\qua$ which is the maximum absolute value of $\qua(\vad)$ over $B_\infty(0;\Delta)$, we see that Lemma~\ref{quadraticL2errors} establishes a relation between two norms of $\qua$. By explicitly deriving constants in terms of $\card(\alpha)$ we strengthen the bounds in  Lemma~\ref{quadraticL2errors} for the cases of sparse models.
\end{remark}

We are now ready to present our main result.

\begin{theorem}\label{maintheorem}
Let Assumption~\ref{assumptionhessiansparse-app} hold (approximate Hessian sparsity).
Given $\Delta$ and $x_0$, let~$h$ be the corresponding sparsity level of the fully quadratic model
guaranteed by Assumption~\ref{assumptionhessiansparse-app}.
Let  $\Pon=\{\pon^1,...,\pon^p\}$ be a given set of $p$ random points, chosen with respect to the uniform measure in $B_\infty(0;\Delta)$, with
\begin{equation}\label{maintheoremcond}\frac{p}{\log p} \; \geq \; 9c_{partial}\left(h+n+1\right)(\log\left(h+n+1\right))^2\log q,\end{equation}
for some universal constant $c_{partial}>0$,  then with probability larger than $1-n^{-\gamma\log p}$,
for some universal constant $\gamma>0$, the quadratic $m^*(x)=\tilde{m}^*(x-x_0)$, where~$\tilde{m}^*$ is
the solution to the $\ell_1$-minimization problem~(\ref{minl1HC0}),
is a fully quadratic model of $\fo$  on $B_\infty(x_0;\Delta)$ with
 $\nu_2^{m^*}=0$ and constants $\kappa_{ef}$, $\kappa_{eg}$, and $\kappa_{eh}$ not depending on $x_0$ and~$\Delta$.
\end{theorem}

\begin{proof}
From Assumption~\ref{assumptionhessiansparse-app},
 there exists  a fully quadratic model  $\qua$ for $\fo$ on $B_\infty(x_0;\Delta)$ with $\nu_2^m=0$ and some
 constants $\kappa'_{ef}$, $\kappa'_{eg}$, and $\kappa'_{eh}$.
The quadratic polynomial $\tilde{m}(z) = \qua(z+x_0)$, $z \in B_\infty(0;\Delta)$,
satisfies the assumptions of Corollary~\ref{rauhutnoisyHC} and, for the purpose
of the proof, is the quadratic that will be approximately recovered. Now, since $\qua$ is a fully quadratic model,
we have $|\fo(\pon^i+x_0)-\qua(\pon^i+x_0)|\leq \kappa'_{ef}\Delta^{3}$, $y^i \in B_\infty(0;\Delta)$. Therefore
\[ \|\fo(\Pon + x_0)-\qua(\Pon + x_0)\|_2\; \leq \; \sqrt{p}\,\kappa'_{ef}\Delta^3,\]
for any sample set~$Y \subset B_\infty(0;\Delta)$ (where $\kappa'_{ef}$ is independent of~$x_0$ and~$\Delta$).
Note that one can only recover $\qua$ approximately given that the values of $\qua(\Pon+x_0)\simeq\fo(\Pon+x_0)$ are `noisy'.

Consider again $\tilde{m}(z) = m(z+x_0)$,  $z \in B_\infty(0;\Delta)$.
Then, by Corollary~\ref{rauhutnoisyHC}, with probability larger than $1-n^{-\gamma\log p}$,
for a universal constant $\gamma>0$, the solution~$\tilde{m}^\ast$ to the
$\ell_1$-minimization problem~(\ref{minl1HC0}) with $\eta=\sqrt{p}\,\kappa'_{ef}\Delta^3$ satisfies
$$\|\alpha^\ast-\alpha\|_2 \; \leq \; c_{partial} \kappa'_{ef}\Delta^3,$$
where $\coe^\ast$ and $\coe$ are the coefficients of $\tilde{m}^\ast$ and $\tilde{m}$ in the basis $\psi$ given by~(\ref{basehipercubo}), respectively.
Note that $c_{partial}$ does not depend on~$x_0$.
So, by Lemma~\ref{quadraticL2errors},
\begin{eqnarray*}
\left|\tilde{m}^\ast(z)-\tilde{m}(z)\right|&\leq&c_{partial} \left(3\sqrt{\card(\alpha^\ast-\alpha)}\right)\kappa'_{ef}\Delta^3,\\
\left\|\nabla \tilde{m}^\ast(z)-\nabla \tilde{m}(z)\right\|_2&\leq&c_{partial} \left(3\sqrt{5}\sqrt{\card(\alpha^\ast-\alpha)}\right)\kappa'_{ef}\Delta^2,\\
\left\|\nabla^2 \tilde{m}^\ast(z)-\nabla^2 \tilde{m}(z)\right\|_2&\leq&c_{partial} \left(3\sqrt{10}\sqrt{\card(\alpha^\ast-\alpha)}\right)\kappa'_{ef}\Delta,
\end{eqnarray*}
for all $z \in B_\infty(0;\Delta)$.
Note that $\coe^\ast$ and $\coe$ depend on $x_0$ but $\card(\alpha^\ast-\alpha)$ can be easily bounded
independently of $x_0$.
Let now $m^\ast(x) = \tilde{m}^\ast(x-x_0)$ for $x \in B_\infty(x_0;\Delta)$.
Therefore, from the fact that $m$ is fully quadratic (with constants $\kappa'_{ef}$, $\kappa'_{eg}$, and $\kappa'_{eh}$), one has
\begin{eqnarray*}
\left|\qua^\ast(x)-\fo(x)\right|&\leq&\left(c_{partial}\left(3\sqrt{\card(\alpha^\ast-\alpha)}\right)\kappa'_{ef}+\kappa'_{ef}\right)\Delta^3,\\
\left\|\nabla\qua^\ast(x)-\nabla\fo(x)\right\|_2&\leq&\left(c_{partial}\left(3\sqrt{5}\sqrt{\card(\alpha^\ast-\alpha)}\right)\kappa'_{ef}+\kappa'_{eg}\right)\Delta^2,\\
\left\|\nabla^2\qua^\ast(x)-\nabla^2\fo(x)\right\|_2&\leq&\left(c_{partial}\left(3\sqrt{5}\sqrt{\card(\alpha^\ast-\alpha)}\right)\kappa'_{ef}+\kappa'_{eh}\right)\Delta,
\end{eqnarray*}
for all $x \in B_\infty(x_0;\Delta)$.

Since $\qua^\ast$ is a quadratic function, its Hessian is Lipschitz continuous with Lipschitz constant~$0$,
so one has that $\nu_2^{m^*}=0$. Hence $\qua^\ast$ is a fully quadratic model of $\fo$ on $B_\infty(x_0;\Delta)$.
\end{proof}

Note that the result of Theorem~\ref{maintheorem} is obtained for a number $p$ of sampling points satisfying (see~(\ref{maintheoremcond}) and recall that $q=\OOO(n^2)$)
\[
\frac{p}{\log p} \; = \; \OOO(n(\log n)^3)
\]
when $h=\OOO(n)$, i.e., when the number of non-zero elements of the Hessian of $\fo$ at~$x_0$ is of the order of $n$. Since $p<\dimquah$, one obtains
\begin{equation}\label{kigualOlog4nn}
p \; = \; \OOO\left(n(\log n)^4\right).
\end{equation}

Theorem~\ref{maintheorem} does not directly assume Hessian sparsity of $\fo$.
It is worth observing again that Theorem~\ref{maintheorem} can be established
under no conditions on the sparsity {\em pattern} of the Hessian of $\fo$.

Problem~(\ref{minl1HC0}) is a second order cone programming problem~\cite{AlizadehGoldfarb2003} and can, hence, be solved in polynomial time.
However it is typically easier in practice to solve linear programming problems.
Since the second order Taylor model $\tay$ satisfies
$\left\|\tay(\Pon+x_0)-\fo(\Pon+x_0)\right\|_\infty\leq\eta/\sqrt{p}$
(where $\eta=\sqrt{p}\,\kappa'_{ef} \Delta^3$), because $\tay$ is fully quadratic for $\fo$, instead of~(\ref{minl1HC0}), one can consider
\[\begin{array}{ll}
\min & \left\|\coe_Q^\qua\right\|_1\\ [1ex]
\operatorname{s.t.} & \left\|M(\psi_Q,\Pon)\coe_Q^\qua+M(\psi_L,\Pon)\coe_L^\qua-\fo(\Pon+x_0)\right\|_\infty \; \leq \; \frac1{\sqrt{p}}\eta,
\end{array}\]
 which is a linear program. In our implementation, as we will discuss in the next section, we chose
to  impose the interpolation constraints exactly which corresponds to setting~$\eta=0$ in the above formulations, hence simplifying
 parameter choices. Also, recent work has provided some insight
 for  why this choice works well (see~\cite{PWojtaszczyk_2010}).

Theorem~\ref{maintheorem}
cannot strictly validate a practical setting in DFO like the one discussed in the next section.
It serves to  provide motivation and insight
on the use of $\ell_1$-minimization to build underdetermined
quadratic models for functions with sparse Hessians.
 It also is the first result, to our knowledge, that establishes a reasonable approach to building fully quadratic models with underdetermined interpolation, when the sparsity structure of the objective function is not known.
However, in the current implementation   the sampling is done deterministically
in order to be able to reuse existing sample points. This may be lifted in future parallel implementations.
Note that the constants in the bound~(\ref{maintheoremcond}) (and thus in~(\ref{kigualOlog4nn}))
render the current bounds impractical. In fact, the best known upper bound
(see~\cite{HRauhut_2010}) for the universal constant $c_{total}$ appearing in~(\ref{herefortheremark}) is $c_{total}<17190$
($c_{partial}$ is of the same order),
making~(\ref{maintheoremcond}) only applicable if~$n$ is much greater
than the values for which DFO problems are tractable today by deterministic algorithms.
However, such a bound is most likely not
tight;  in fact, similar universal constants appearing in the setting of
compressed sensing are known to be much smaller in practice.

\section{A practical interpolation-based trust-region method}\label{chapter5}

\subsection{Interpolation-based trust-region algorithms for DFO}

Trust-region methods are a well known class of algorithms for the numerical
solution of nonlinear programming
problems~\cite{ARConn_NIMGould_PhLToint_2000a,JNocedal_SJWright_2006}.
In this section we will give a brief summary of these methods
when applied to the unconstrained minimization
of a smooth function $\fo:D \subset \RR^n \rightarrow \RR$,
\begin{equation}\label{problemaglobal}\min_{x \in \RR^n} \; \fo(x),\end{equation}
without using the derivatives of the objective function $\fo$. For comprehensive coverage we refer the reader to \cite{ARConn_KScheinberg_LNVicente_2009}.

At each iteration~$k$, these methods build a model $m_k(x_k+s)$ of
the objective function in a {\it trust region} of the form $B_p(x_\nstep;\Delta_\nstep)$,
typically with $p=2$,  around the
current iterate $x_k$. The scalar $\Delta_k$ is then called the trust-region radius.
A step $s_k$  is determined by solving the trust-region subproblem
\begin{equation}\label{trustregoinsubproblem}
\min_{s \in B_2(0;\Delta_\nstep)} \; m_\nstep(x_k+s).
\end{equation}
Then, the value $\fo(x_k+s_k)$ is computed and the actual reduction in the objective function
($ared_k = f(x_k) - f(x_k+s_k)$) is compared to
the predicted reduction in the model ($pred_k = m_k(x_k) - m_k(x_k+s_k)$).
If the ratio is big
enough  ($\rho_k=ared_k/pred_k \geq \eta_1 \in (0,1)$),
then $x_k+s_k$ is accepted as the new iterate and  the trust-region radius
may be increased. Such iterations are called successful.
If the ratio is small ($\rho_k < \eta_1$), then the step is
rejected and the trust-region radius is decreased. Such iterations
are called unsuccessful.

The global convergence properties of these methods are strongly dependent
on the requirement that, as the trust region becomes smaller,
the model becomes more accurate, implying in particular that the
 trust-region radius
is bounded away from zero, as long as the stationary point is not reached.
Taylor based-models,
when derivatives are known, naturally satisfy this requirement.
However, in the DFO setting, some provision has to be taken in the model
and sample set management to ensure global convergence. These provisions
aim at guaranteeing that the models produced by the algorithm are fully linear or fully
quadratic, and to guarantee global convergence this must be done in arbitrarily
smaller trust regions thus driving the trust-region radius to zero (such a
procedure is ensured by the so-called criticality step, which has indeed
be shown necessary~\cite{KScheinberg_PhLToint_2009}).

 Conn, Scheinberg, and Vicente~\cite{ARConn_KScheinberg_LNVicente_2009-paper}
proved global convergence to first and second order stationary points depending
 whether fully linear or fully quadratic models are used. The approach proposed in~\cite{ARConn_KScheinberg_LNVicente_2009-paper} involves special model improving iterations.
Scheinberg and Toint~\cite{KScheinberg_PhLToint_2009}
have recently shown global convergence to first order stationary points for their {\it self-correcting
 geometry} approach which replaces model-improving iterations by an appropriate
 update of the sample set using only the new trust-region iterates.

Our results derived in Section \ref{chapter4} provide us with a new method to
produce (with high probability) fully quadratic models by considering randomly sampled sets, instead of model handling iterations as is done in~\cite{ARConn_KScheinberg_LNVicente_2009-paper}  and~\cite{KScheinberg_PhLToint_2009}.
On the other hand, to develop full convergence theory of DFO methods using randomly sampled sets, one needs to adapt the convergence proofs used in~\cite{ARConn_KScheinberg_LNVicente_2009-paper}  and~\cite{KScheinberg_PhLToint_2009} to the case where successful iterations are guaranteed  with high probability, rather than deterministically. This is a subject for future research.

\subsection{A practical interpolation-based trust-region method}

We now introduce a simple practical algorithm which we chose for testing
the performance of different underdetermined models. This algorithm
follows some of the basic
ideas of the approach introduced by
Fasano, Morales, and Nocedal~\cite{GFasano_JLMorales_JNocedal_2009},
which have also inspired the authors in~\cite{KScheinberg_PhLToint_2009}.
The quality of the sample sets is maintained in its simplest form --- simply ensuring sufficient number of sample points ($n+1$ or more) in a reasonable proximity from the current iterate. This approach is theoretically weak as shown
in~\cite{KScheinberg_PhLToint_2009}, but seems to work well in practice.

Unlike~\cite{GFasano_JLMorales_JNocedal_2009},
we discard the sample point farthest away from the new iterate
(rather than the sample point farthest away from the current iterate).
Also, in~\cite{GFasano_JLMorales_JNocedal_2009}, only determined quadratic models were built based on  $p_{\max} = (n+1)(n+2)/2$ sample points.
We compare approaches that use minimum Frobenius or $\ell_1$ norm interpolation to build the models and hence we allow sample sets
of any size less than or equal to $p_{\max}$.
This poses additional issues to those considered in~\cite{GFasano_JLMorales_JNocedal_2009}. For instance,
until the cardinality of the sample set reaches~$p_{\max}$,
we do not discard points from the sample set and always add new trial points
independently of whether or not they are accepted as new iterates, in an attempt to be as greedy as possible when taking advantage of function evaluations.

Another difference from~\cite{GFasano_JLMorales_JNocedal_2009} is that we discard points that are too
far from the current iterate when the trust-region radius becomes small (this can be viewed as a weak
criticality condition), expecting that the next iterations will refill the sample set resulting in a
similar effect as a criticality step.
Thus, the cardinality of our sample set might fall below $p_{\min} = n+1$, the number required to build fully linear models in general. In such situations, we never reduce the trust-region radius.

\noindent
\begin{algorithm}[A practical DFO trust-region algorithm]
\label{alg:dfo}
\ \\
\noindent
{\bf Step 0: Initialization.}
\begin{itemize}
\item[] {\bf Initial values.} Select values for the constants $\epsilon_g(=10^{-5}) > 0$,
$\delta(=10^{-5}) > 0$, $0 < \eta_1(=10^{-3})$, $\eta_2 (=0.75) > \eta_1$,
and $0 < \gamma_1(=0.5) < 1 < \gamma_2(=2)$.
Set $p_{\min} = n+1$ and $p_{\max} = (n+1)(n+2)/2$.
Set the initial trust-region radius $\Delta_0(=1) > 0$.
Choose the norm $t=1$ (the $\ell_1$-norm) or $t=2$ (the Frobenius norm).

\item[] {\bf Initial sample set.}
Let the starting point $x_0$ be given.
Select as an initial sample set
$Y_0 = \{ x_0, x_0 \pm \Delta_0 e_i, i=1,\ldots,n \}$, where the $e_i$'s are the columns of the
identity matrix of order~$n$.

\item[] {\bf Function evaluations.}
Evaluate the objective function at all $y \in Y_0$.

\end{itemize}
Set $k=0$.

\noindent
{\bf Step 1: Model building.}
\begin{itemize}
\item[]
Form a quadratic model $m_k(x_k+s)$ of the objective function
from $Y_k$.
 Solve the problem
\begin{equation}\label{modelo52}
\begin{array}{cl}
\min & \frac1p\|\coe_Q\|_t^t\\ [1ex]
\operatorname{s.t.} & M(\bar{\phi}_Q,Y_k)\coe_Q+M(\bar{\phi}_L,Y_k)\coe_L \; = \; \fo(Y_k),
\end{array}
\end{equation}
where $\coe_Q$ and $\coe_L$ are, respectively, the coefficients of
order~2 and order less than~2 of the model.
\end{itemize}

\noindent
{\bf Step 2: Stopping criteria.}
\begin{itemize}
\item[]
Stop if $\| g_k \| \leq \epsilon_g$ or $\Delta_k \leq \delta$.
\end{itemize}

\noindent
{\bf Step 3: Step calculation.}
\begin{itemize}
\item[] Compute a step $s_k$ by solving (approximately) the trust-region
subproblem~(\ref{trustregoinsubproblem}).
\end{itemize}

\noindent
{\bf Step 4: Function evaluation.}
\begin{itemize}
\item[]
Evaluate the objective function at $x_k+s_k$. Compute
 $\rho_k=(f(x_k) - f(x_k+s_k))/(m_k(x_k) - m_k(x_k+s_k))$.
\end{itemize}

\noindent
{\bf Step 5: Selection of the next iterate and trust radius update.}
\begin{itemize}
\item[] If $\rho_k < \eta_1$, reject the trial step, set $x_{k+1} = x_k$,
and reduce the trust-region radius, if $|Y_k| \geq p_{\min}$, by setting $\Delta_k = \gamma_1 \Delta_k$
({\bf unsuccessful iteration}).
\item[] If $\rho_k \geq \eta_1$, accept the trial step $x_{k+1} = x_k + s_k$
({\bf successful iteration}).
\item[] Increase the trust-region radius,
$\Delta_{k+1} = \gamma_2 \Delta_k$, if $\rho_k > \eta_2$.
\end{itemize}

\noindent
{\bf Step 6: Update the sample set.}
\begin{itemize}
\item[]
If $|Y_k| = p_{\max}$, set $y^{out}_{k} \in \argmax \| y - x_{k+1} \|_2$ (break ties arbitrarily).
\item[]
If the iteration was successful:
\begin{itemize}
%\item[] Compute $y^{out}_{k} \in \argmax \| y - x_{k+1} \|$ (break ties arbitrarily).
\item[] If $|Y_k| = p_{\max}$,
$Y_{k+1} = Y_k \cup \{ x_{k+1} \} \setminus
\{ y^{out}_{k} \}$.
\item[]
If $|Y_k| < p_{\max}$, $Y_{k+1} = Y_k \cup \{ x_{k+1} \}$.
\end{itemize}
If the iteration was unsuccessful:
\begin{itemize}
%\item[] Compute $y^{out}_{k} \in \argmax \| y - x_k \|$ (break ties arbitrarily).
\item[] If $|Y_k| = p_{\max}$,
$Y_{k+1} = Y_k \cup \{ x_k+s_k \} \setminus
\{ y^{out}_{k} \}$ if $\| (x_k+s_k) - x_k \|_2 \leq \| y^{out}_{k} - x_k \|_2$.
\item[]
If $|Y_k| < p_{\max}$, $Y_{k+1} = Y_k \cup \{ x_k+s_k \}$.
\end{itemize}
\end{itemize}

\noindent
{\bf Step 7: Model improvement.}
\begin{itemize}
\item[] When $\Delta_{k+1} < 10^{-3}$, discard from
$Y_{k+1}$ all the points outside $B(x_{k+1}; r \Delta_{k+1})$,
where $r$ is chosen as the smallest number in $\{100,200,400,800,...\}$
for which at least three sample points from $Y_{k+1}$
are contained in $B(x_{k+1}; r \Delta_{k+1})$.
\end{itemize}
Increment $k$ by $1$ and return to Step~1.

\end{algorithm}

We note that relying on the model gradient to stop might not be a reliable
stopping criterion, as we need to resample the function inside a smaller
trust region and construct a new model before safely quitting (criticality step).
However, in a practical method, one avoids doing it.

\subsection{Numerical results}\label{sec53}

In this section we  describe  the numerical
experiments which  test the performance
of Algorithm~\ref{alg:dfo} implemented in MATLAB.
 In particular we are  interested in
testing two variants of Algorithm~\ref{alg:dfo} defined by the
norm used to compute the model in~(\ref{modelo52}). The first
variant makes use of the $\ell_2$-norm and leads to
minimum Frobenius norm models.
The solution of~(\ref{modelo52}) with $t=2$ is a convex quadratic problem subject to equality constraints and hence is equivalent to solving
 the following linear system
\[
\left[\begin{array}{cc}
M(\bar{\phi}_Q,\Pon_k)M(\bar{\phi}_Q,\Pon_k)^T        &          M(\bar{\phi}_L,\Pon_k)\\
M(\bar{\phi}_L,\Pon_k)^T                            &                 0
\end{array}\right]
\left[\begin{array}{c}\lambda\\ \coe_L\end{array}\right] \; = \;
\left[\begin{array}{c}\fo(\Pon_k)\\ 0\end{array}\right]
\]
with $\alpha_Q = M(\bar{\phi}_Q,\Pon_k)^T\lambda$.

We solved this system using SVD, regularizing extremely small singular values
after the decomposition and before
performing the backward solves, in an attempt to remediate
extreme
ill-conditioning caused by nearly ill-poised sample sets.
The second approach consisted in using $t=1$, leading
to minimum $\ell_1$-norm
models and attempting to recover sparsity in the Hessian of the objective
function. To solve problem~(\ref{modelo52}) with $t=1$ we formulated it first as a linear program.
In both cases, $t=1,2$, we first scaled the corresponding problems
by shifting the sample set to the origin (i.e., translating all
the sample points such that the current iterate
coincides with the origin) and then scaling the points so that they
lie in $B_2(0;1)$ with at least one scaled point at the border
of this ball.
This procedure, suggested in~\cite[Section 6.3]{ARConn_KScheinberg_LNVicente_2009}, leads to an improvement
of the numerical results, especially in the minimum Frobenius
norm case.

The trust-region subproblems~(\ref{trustregoinsubproblem}) have been solved using
the routine {\tt trust.m} from the MATLAB Optimization Toolbox
which corresponds essentially to the algorithm of Mor\'e and
Sorensen~\cite{JJMore_DCSorensen_1983}. To solve the linear
programs~(\ref{modelo52}), with $t=1$, we have used the routine {\tt linprog.m}
from the same MATLAB toolbox. In turn, {\tt linprog.m}
uses in most of the instances considered in our optimization
runs the interior-point solver {\tt lipsol.m},
developed by Zhang~\cite{YZhang_1997-tech}.

In the first set of experiments,
we considered the test set of unconstrained problems
from the CUTEr collection~\cite{MGould_DOrban_PhLToint_2003} used in~\cite{SGratton_PLToint_ATro_2010}, and in~\cite{GFasano_JLMorales_JNocedal_2009}.
We used the same dimension choices as in~\cite{SGratton_PLToint_ATro_2010}
but we removed all problems considered there with less than~5 variables.
This procedure resulted in the test set described in
Table~\ref{tab:26problemas}. Most of these problems exhibit some
form of sparsity in the Hessian of the objective function,
 for instance, a banded format.

\begin{table}[h]
\begin{center}
\scriptsize
\begin{tabular}{|l|c|c|c|c|} \hline
problem & $n$ & {\tt NNZH} & {\tt DFO-TR Frob} ($acc=6$) & {\tt DFO-TR l1} ($acc=6$)\\ \hline
{\tt ARGLINB } & 10 & 55 & 57       & 59\\ \hline
{\tt ARGLINC } & 8  & 21 & 56       & 57\\ \hline
{\tt ARWHEAD } & 15 & 29 & 195      & 143\\ \hline
{\tt BDQRTIC } & 10 & 40 & 276      & 257\\ \hline
{\tt BIGGS6  } &  6 & 21 & 485      & 483\\ \hline
{\tt BROWNAL } & 10 & 55 & 437      & 454\\ \hline
{\tt CHNROSNB} & 15 & 29 & 993      & 1004 \\ \hline
{\tt CRAGGLVY} & 10 & 19 & 548      & 392\\ \hline
{\tt DIXMAANC} & 15 & 44 & 330      & 515\\ \hline
{\tt DIXMAANG} & 15 & 44 & 395      & 451\\ \hline
{\tt DIXMAANI} & 15 & 44 & 429      & 361\\ \hline
{\tt DIXMAANK} & 15 & 44 & 727      & 527\\ \hline
{\tt DIXON3DQ} & 10 & 18 & --       & --\\ \hline
{\tt DQDRTIC } & 10 & 10 & 25       & 25\\ \hline
{\tt FREUROTH} & 10 & 19 & 249      & 252\\ \hline
{\tt GENHUMPS} &  5 & 9 & 1449     & 979\\ \hline
{\tt HILBERTA} & 10 & 55 & 8        & 8\\ \hline
{\tt MANCINO } & 10 & 55 & 106      & 73\\ \hline
{\tt MOREBV  } & 10 & 27 & 111      & 105\\ \hline
{\tt OSBORNEB} & 11 & 66 & 1363     & 1023\\ \hline
{\tt PALMER1C} &  8 & 36 & --       & -- \\ \hline
{\tt PALMER3C} &  8 & 36 & 56       & 53\\ \hline
{\tt PALMER5C} &  6 & 21 & 29       & 29\\ \hline
{\tt PALMER8C} &  8 & 36 & 60       & 55\\ \hline
{\tt POWER   } & 10 & 55 & 466      & 428\\ \hline
{\tt VARDIM  } & 10 & 55 & 502      & 314\\ \hline
\end{tabular}
\caption{\label{tab:26problemas}The test set used in the first set of experiments and the
corresponding dimensions (first three columns). The third column reports the upper bound provided
by CUTEr on the number of nonzero elements of the Hessian stored using the coordinate format.
 The last two columns report the total number of function evaluations required by
 Algorithm~\ref{alg:dfo} to achieve an accuracy of $10^{-6}$ on the objective function value
(versions {\tt DFO-TR Frob} and {\tt DFO-TR l1}). Both approaches failed to
solve  two of the problems.}
\normalsize
\end{center}
\end{table}

In order to present the numerical results for all problems and all
methods (and variants) considered, we have used the so-called performance profiles, as suggested in~\cite{EDDolan_JJMore_2002}.
Performance profiles are, essentially, plots of cumulative distribution
functions $\rho(\tau)$ representing a performance ratio for the different solvers.
Let $\mathcal{S}$
be the set of solvers and $\mathcal{P}$ the set of problems.
Let $t_{p,s}$ denote the performance of the solver~$s \in \mathcal{S}$ on the problem~$p \in \mathcal{P}$ ---
lower values of $t_{p,s}$ indicate better performance.
This performance ratio $\rho(\tau)$ is defined by first setting
$r_{p,s}=t_{p,s}/\min\{t_{p,\bar{s}}: \bar{s}\in \mathcal{S}\}$, for $p\in\mathcal{P}$ and $s\in \mathcal{S}$.
Then, one defines $\rho_s(\tau)=(1/|\mathcal{P}|)|\{p\in \mathcal{P}:r_{p,s}\leq\tau\}|$. Thus,  $\rho_s(1)$ is the probability
that  solver~$s$ has  the best performance among all solvers. If we are only interested in determining
which solver is the most efficient (is the fastest on most problems),
then we should compare the values of
$\rho_s(1)$ for all the solvers. On the other hand, solvers with the largest value of  $\rho_s(\tau)$ for large $\tau$ are the ones which
 solve the largest number of problems in $\mathcal{P}$, hence are the most robust.
We are interested in considering a wide range of values for $\tau$, hence, we plot the performance profiles in a $\log$-scale (now, the value at $0$ represents the probability of winning over the other solvers).

\begin{figure}[!h]
\begin{center}
\includegraphics[width=10cm]{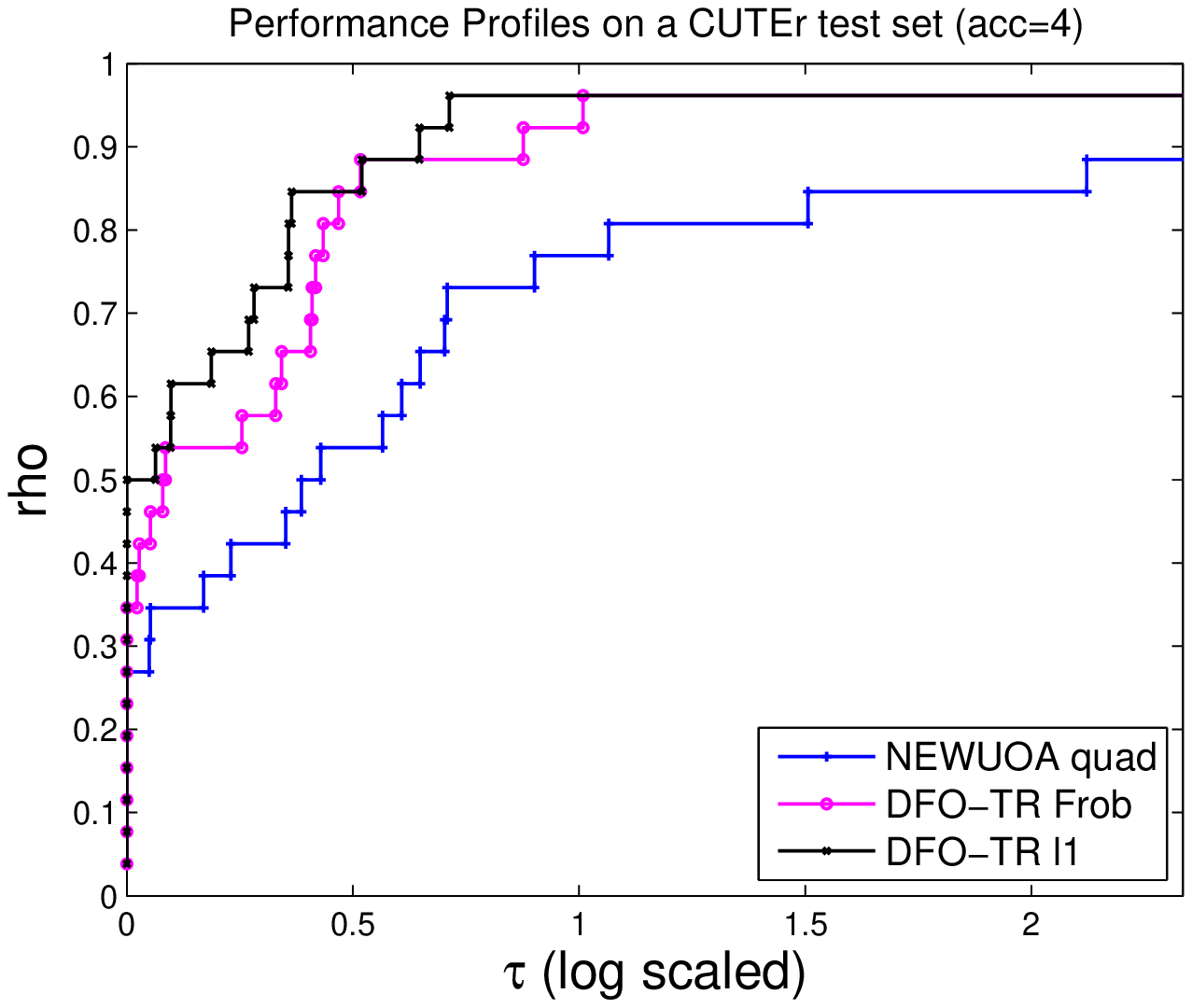}\\
\includegraphics[width=10cm]{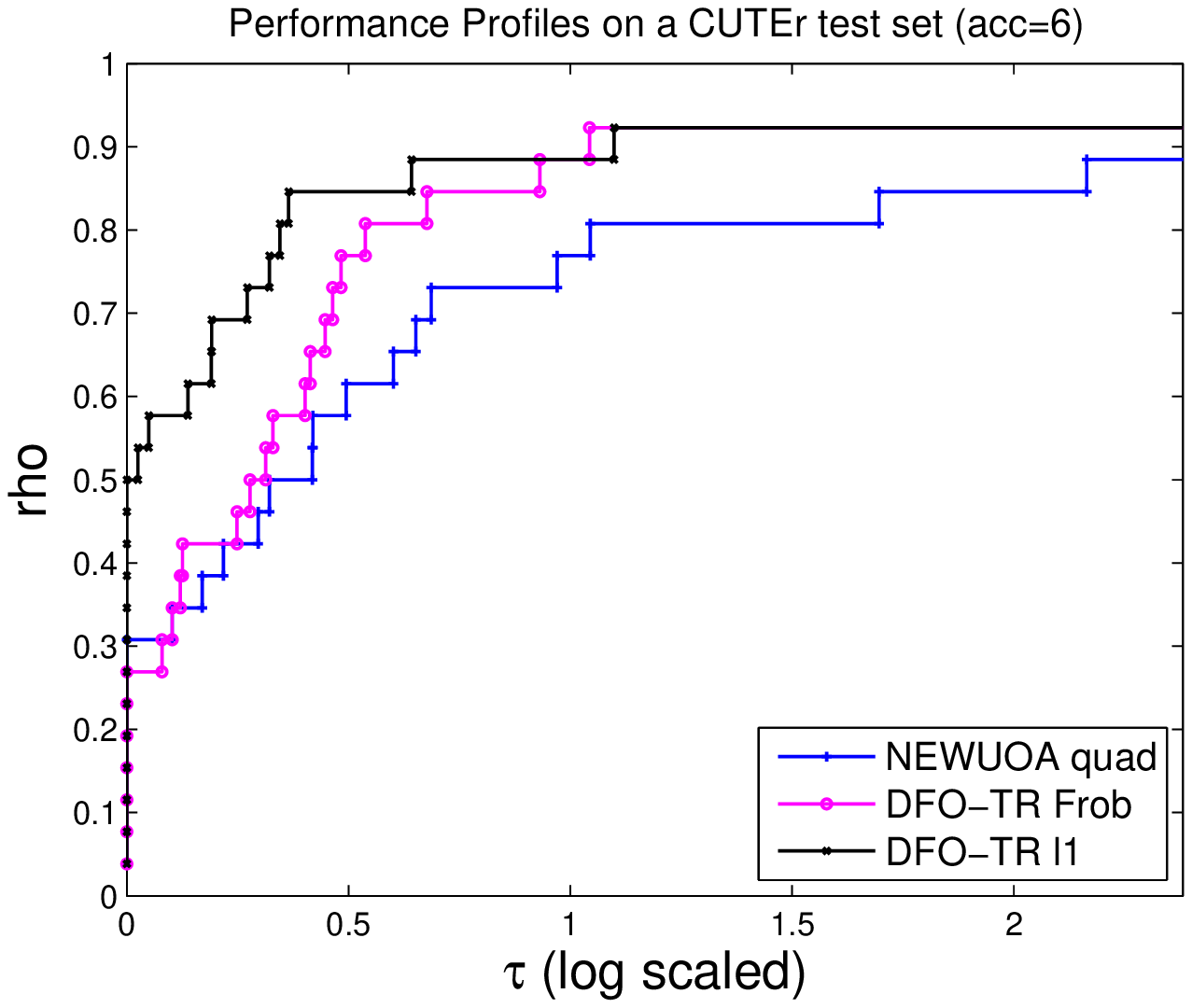}\\
\caption{Performance profiles comparing Algorithm~\ref{alg:dfo}
(minimum Frobenius and $\ell_1$ norm versions)
and NEWUOA~\cite{MJDPowell_2003,MJDPowell_2004-tech}, on the test set of Table~\ref{tab:26problemas},
for two levels of accuracy ($10^{-4}$ above and
$10^{-6}$ below).}\label{performanceprofiles}
\end{center}
\end{figure}

In our experiments, we took the \textit{best}
objective function value from~\cite{SGratton_PLToint_ATro_2010}
(obtained by applying a derivative-based Non-Linear Programming solver), as a benchmark to detect
whether a problem was successfully solved up to a
certain accuracy $10^{-acc}$. The number $t_{p,s}$ is then the number of
function evaluations needed to achieve an
objective function value within an absolute error of $10^{-acc}$
of the best objective function value;
otherwise a failure occurs and the value of $r_{p,s}$ used to build
the profiles is set to a  large number
(see~\cite{EDDolan_JJMore_2002}). Other measures of performance could be used for $t_{p,s}$ but the number of function evaluations is
 the most appropriate for expensive objective functions.
In Figure~\ref{performanceprofiles}, we plot performance profiles for the two variants
of Algorithm~\ref{alg:dfo} mentioned above and for the state-of-the-art solver NEWUOA~\cite{MJDPowell_2003,MJDPowell_2004-tech}.
Following~\cite{EDolan_JMore_TMunson_2006},
and in order to provide a fair comparison,
solvers are run first with their own default
stopping criterion and if convergence can not be declared
another run is repeated with tighter tolerances.
In the case of Algorithm~\ref{alg:dfo}, this procedure led to
$\epsilon_g = \delta = 10^{-7}$ and a maximum number of
$15000$ function evaluations. For NEWUOA we used the data prepared for~\cite{SGratton_PLToint_ATro_2010}
also for a maximum number of
$15000$ function evaluations.

Note that NEWUOA requires
an interpolation of fixed cardinality
in the interval $[2n+1,(n+1)(n+2)/2]$ throughout the entire
optimization procedure.
We looked at the extreme possibilities, $2n+1$ and $(n+1)(n+2)/2$,
and are reporting results only with the latter one
({\tt NEWUOA quad} in the plots)
since it was the
one which gave the best results.
The two variants of Algorithm~\ref{alg:dfo}, are referred
 to as {\tt DFO-TR Frob}
 (minimum Frobenius norm models)
and {\tt DFO-TR l1} (minimum $\ell_1$-norm models).
Two levels of accuracy ($10^{-4}$ and $10^{-6}$) are considered
in Figure~\ref{performanceprofiles}.
One can observe that {\tt DFO-TR l1} is the most efficient version ($\tau=0$ in the $\log$ scale) and basically as robust as the {\tt DFO-TR Frob} version (large values of $\tau$), and that both versions of the Algorithm~\ref{alg:dfo} seem to outperform {\tt NEWUOA quad} in efficiency and robustness.

\begin{table}[h]
\begin{center}
\scriptsize
\begin{tabular}{|l|c|c|c|} \hline
problem & $n$ & type of sparsity & {\tt NNZH} \\ \hline
{\tt ARWHEAD}     & 20 & sparse & 39 \\ \hline
{\tt BDQRTIC}     & 20 & banded & 90\\ \hline
{\tt CHNROSNB}     & 20 & banded & 39\\ \hline
{\tt CRAGGLVY}    & 22 & banded & 43 \\ \hline
{\tt DQDRTIC}     & 20 & banded & 20 \\ \hline
{\tt EXTROSNB}    & 20 & sparse & 39 \\ \hline
{\tt GENHUMPS}    & 20 & sparse & 39 \\ \hline
{\tt LIARWHD}     & 20 & sparse & 39 \\ \hline
{\tt MOREBV}      & 20 & banded & 57 \\ \hline
{\tt POWELLSG}    & 20 & sparse & 40 \\ \hline
{\tt SCHMVETT}    & 20 & banded & 57 \\ \hline
{\tt SROSENBR}    & 20 & banded & 30 \\ \hline
{\tt WOODS}       & 20 & sparse & 35 \\ \hline
\end{tabular}
\caption{\label{tab:14problemas}The test set used in the second set of experiments.
For each problem we included the number of variables and the type of sparsity, as described in~\cite{BColson_PhLToint_2005}.
The last column reports the upper bound provided
by CUTEr on the number of nonzero elements of the Hessian stored using the coordinate format.}
\normalsize
\end{center}
\end{table}

In the second set of experiments we ran Algorithm~\ref{alg:dfo}
for the two variants (minimum Frobenius and $\ell_1$ norm models)
on the test set of CUTEr unconstrained problems used
in the paper~\cite{BColson_PhLToint_2005}. These problems
are known to have a significant amount of sparsity in
the Hessian (this information as well as the dimensions
selected is described in Table~\ref{tab:14problemas}). We used
 $\epsilon_g = \delta = 10^{-5}$ and a maximum number of
$5000$ function evaluations. In Table~\ref{tab:14resultados}, we report the
number of objective function evaluations taken as well as the
final objective function value obtained.
In terms of function evaluations,
one can observe that {\tt DFO-TR l1}   wins in approximately $8/9$ cases,
when compared to the {\tt DFO-TR Frob} version, suggesting that the former
is more efficient than the latter in the presence of Hessian sparsity.
Another interesting aspect of the
{\tt DFO-TR~l1} version is some apparent ability to produce the
final model
with gradient of smaller norm.

\begin{table}[h]
\begin{center}
\scriptsize
\begin{tabular}{|l|c|c|c|c|} \hline
problem & {\tt DFO-TR Frob/l1} &
                     $\#$ $\fo$ eval &
                                     final $\fo$ value  &
                                                    final $\nabla m$ norm\\\hline
{\tt ARWHEAD}     & {\tt Frob}      & 338  & 3.044e-07 & 3.627e-03 \\
{\tt ARWHEAD}     & {\tt l1}        & 218  & 9.168e-11 & 7.651e-07 \\\hline
{\tt BDQRTIC}     & {\tt Frob}      & 794  & 5.832e+01 &5.419e+05\\
{\tt BDQRTIC}     & {\tt l1}        & 528  & 5.832e+01 &6.770e-02\\\hline
{\tt CHNROSNB}    & {\tt Frob}      & 2772 &  3.660e-03  &2.025e+03\\
{\tt CHNROSNB}    & {\tt l1}        & 2438 &  2.888e-03  &1.505e-01\\\hline
{\tt CRAGGLVY}    & {\tt Frob}      & 1673 &  5.911e+00  &1.693e+05\\
{\tt CRAGGLVY}    & {\tt l1}        & 958  &  5.910e+00  &8.422e-01\\\hline
{\tt DQDRTIC}     & {\tt Frob}      & 72   &  8.709e-11  &6.300e+05\\
{\tt DQDRTIC}     & {\tt l1}        & 45   &  8.693e-13  &1.926e-06\\\hline
{\tt EXTROSNB}    & {\tt Frob}      & 1068 &  6.465e-02  &3.886e+02\\
{\tt EXTROSNB}    & {\tt l1}        & 2070 &  1.003e-02  &6.750e-02\\\hline
{\tt GENHUMPS}    & {\tt Frob}      & 5000 & 4.534e+05 & 7.166e+02\\
{\tt GENHUMPS}    & {\tt l1}        & 5000 & 3.454e+05 & 3.883e+02\\\hline
{\tt LIARWHD}     & {\tt Frob}      & 905  & 1.112e-12 & 9.716e-06\\
{\tt LIARWHD}     & {\tt l1}        & 744  & 4.445e-08 & 2.008e-02\\\hline
{\tt MOREBV}      & {\tt Frob}      & 539  & 1.856e-04 & 2.456e-03\\
{\tt MOREBV}      & {\tt l1}        & 522  & 1.441e-04 & 3.226e-03\\\hline
{\tt POWELLSG}    & {\tt Frob}      & 1493 & 1.616e-03 & 2.717e+01\\
{\tt POWELLSG}    & {\tt l1}        & 5000 & 1.733e-04 & 2.103e-01\\\hline
{\tt SCHMVETT}    & {\tt Frob}      & 506  & -5.400e+01 &1.016e-02\\
{\tt SCHMVETT}    & {\tt l1}        & 434  & -5.400e+01 &7.561e-03\\\hline
{\tt SROSENBR}    & {\tt Frob}      & 456  & 2.157e-03 & 4.857e-02\\
{\tt SROSENBR}    & {\tt l1}        & 297  & 1.168e-02 & 3.144e-01\\\hline
{\tt WOODS}       & {\tt Frob}      & 5000 & 1.902e-01 & 8.296e-01\\
{\tt WOODS}       & {\tt l1}        & 5000 & 1.165e+01 & 1.118e+01\\\hline
\end{tabular}
\caption{\label{tab:14resultados}Results obtained by {\tt DFO-TR Frob} and {\tt DFO-TR l1} on the problems of
Table~\ref{tab:14problemas} (number of evaluations of the objective function, final value of the objective function, and the norm of the final model gradient).}
\normalsize
\end{center}
\end{table}

\section{Conclusion}\label{chapter6}

Since compressed sensing emerged, it has been deeply connected to
optimization, using optimization as a fundamental tool (in particular, to solve
$\ell_1$-minimization problems). In this paper, however, we have shown that
compressed sensing methodology can also serve as a powerful tool for
optimization, in particular for Derivative-Free Optimization (DFO),
where structure recovery can improve the performance of optimization methods.
Namely, our goal was  to construct
fully quadratic models (essentially models with an accuracy as good
as second order Taylor models; see Definition~\ref{deffullyquadmodel}) of a
function with sparse Hessian using underdetermined quadratic interpolation
on a sample set with potentially much fewer than~$\OOO(n^2)$ points.
We were able to achieve this as is shown  in Theorem~\ref{maintheorem},
 by considering an appropriate polynomial basis and
 random sample sets of only~$\OOO(n (\log n)^4)$ points
when the number of non-zero components of the Hessian is~$\OOO(n)$.
The corresponding quadratic interpolation models were built
by minimizing the $\ell_1$-norm of the entries of the Hessian model.
We then tested the new model selection approach in a  deterministic
setting, by using the  minimum $\ell_1$-norm quadratic models
in a practical interpolation-based trust-region
method (see Algorithm~\ref{alg:dfo}). Our algorithm
was able to outperform state-of-the-art DFO methods as shown
in the numerical experiments reported in Section~\ref{sec53}.

One possible way of solving the $\ell_1$-minimization
problem~(\ref{minl1cap2}) in the context of
interpolation-based trust-region methods
is to rewrite it as a linear program. This approach was used
to numerically test Algorithm~\ref{alg:dfo} when solving
problems~(\ref{modelo52}) for $t=1$.
For problems of up to~$n=20,30$ variables, this
way of solving the $\ell_1$-minimization problems has produced
excellent results in terms of the derivative-free
solution of the original minimization problems~(\ref{problemaglobal})
 and is reasonable in terms of the overall CPU time.

However, for larger values of~$n$, the repeated solution of the
linear programs  introduces significant overhead.
Besides the  increase in the dimension, one also has to
consider possible ill-conditioning arising due to badly poised sample sets.
Although related linear programming problems are solved in consecutive
iterations, it is not trivial to use
warmstart. In fact, the number
of rows in the linear programs change  frequently, making
it difficult to warmstart simplex-based methods.
An alternative is to attempt to approximately solve
problem~(\ref{minl1cap2}) by solving
$\min \| M(\bar{\phi},\Pon) \alpha - f(\Pon) \|_2 + \tau \| \alpha_Q \|_1$
for appropriate values of $\tau > 0$. We conducted
preliminary testing along this avenue
but did not succeed in outperforming the linear
programming approach in any respect.
However, it is out of the scope of this paper a deeper study
of the numerical solution of the $\ell_1$-minimization
problem~(\ref{minl1cap2}) in the context of
interpolation-based trust-region methods.

Although we only considered the most common type of sparsity in
unconstrained optimization (sparsity in the Hessian), it is
straightforward to adapt our methodology to the case where
sparsity also appears in the gradient. In particular, if we
aim at only recovering a sparse gradient or a sparse fully linear model,
one can show that the required number of sample points to do so would be
less that $\mathcal{O}(n)$ and tighten to the level of sparsity.

Finally, we would like to stress that building
accurate quadratic models for functions with sparse
Hessians from function samples could be
of interest outside the field of Optimization.
The techniques and theory developed in Section~\ref{chapter4}
could also be applicable in other settings
of Approximation Theory and Numerical Analysis.

\subsection*{Acknowledgments}
We would like to thank Rachel Ward (Courant Institute of Mathematical Sciences,
NYU) for interesting discussions on compressed sensing and to Anke Tr$\ddot{\text{o}}$ltzsch
(CERFACS, Toulouse)
for providing us assistance with the testing environment of Section~\ref{sec53}. We are also grateful to the anonymous referees for their
 helpful comments on the earlier version of the manuscript.

\small

\bibliographystyle{plain}
\bibliography{ref-dfol1}

\end{document}